\UseRawInputEncoding
\documentclass[12pt, reqno]{amsart}
\usepackage[margin=1in]{geometry}
\usepackage{amssymb,latexsym,amsmath,amscd,amsfonts}
\usepackage{latexsym}
\usepackage[mathscr]{eucal}
\usepackage{bm}
\usepackage{tabularx}
\usepackage{mathptmx}
\usepackage{amssymb}
\usepackage{amsthm}
\usepackage{dcolumn}
\usepackage[all]{xy}
\usepackage{enumitem}
\usepackage[utf8]{inputenc}

\def \qed {\hfill \vrule height6pt width 6pt depth 0pt}
\def\textmatrix#1&#2\\#3&#4\\{\bigl({#1 \atop #3}\ {#2 \atop #4}\bigr)}
\def\dispmatrix#1&#2\\#3&#4\\{\left({#1 \atop #3}\ {#2 \atop #4}\right)}
\newcommand{\beg}{\begin{equation}}
	\newcommand{\eeg}{\end{equation}}
\newcommand{\ben}{\begin{eqnarray*}}
	\newcommand{\een}{\end{eqnarray*}}

\newcommand{\A}{\mathbb{A}}
\newcommand{\CA}{\overline{\mathbb{A}}}
\newcommand{\HS}{\mathcal{H}}
\newcommand{\KS}{\mathcal{K}}
\newcommand{\wA}{\hat{A}}
\newcommand{\wB}{\hat{B}}
\newcommand{\wT}{\hat{T}}
\newcommand{\CB}{\overline{\mathbb{B}}}
\newcommand{\la}{\left\langle}
\newcommand{\ra}{\right\rangle}

\newcommand{\Z}{\mathbb{Z}}

\newtheorem{thm}{Theorem}[section]
\newtheorem{cor}[thm]{Corollary}

\newtheorem{prop}[thm]{Proposition}
\numberwithin{equation}{section} \theoremstyle{definition}

\newtheorem{rem}[thm]{Remark}

\def\textmatrix#1&#2\\#3&#4\\{\bigl({#1 \atop #3}\ {#2 \atop #4}\bigr)}
\def\dispmatrix#1&#2\\#3&#4\\{\left({#1 \atop #3}\ {#2 \atop #4}\right)}
\begin{document}
	\title{Spectral constants for the quantum annulus}
	\author{SOURAV PAL, JAMES E. PASCOE AND NITIN TOMAR}
	
	\address[Sourav Pal]{Mathematics Department, Indian Institute of Technology Bombay,
		Powai, Mumbai - 400076, India.} \email{sourav@math.iitb.ac.in}
		
	\address[James E. Pascoe]{Department of Mathematics, Drexel University, Philadelphia,
USA}
	\email{jep362@drexel.edu}	
	
	\address[Nitin Tomar]{Mathematics Department, Indian Institute of Technology Bombay, Powai, Mumbai-400076, India.} \email{tomarnitin414@gmail.com}		
	
	
		
	\begin{abstract}
		For $\mathbb{A}_r=\{z \in \mathbb{C}: r^{-1}<|z|<r\}$ with $r>1$, we consider the collection \[
		Q\mathbb{A}_r=\{T: T \ \text{is an invertible operator and} \ \|T\|, \|T^{-1}\|\leq r \},
		\] 
		which is referred to as the \textit{quantum annulus}. McCullough-Pascoe \cite{Pas-McCull} proved a dilation theorem for operators in $Q\mathbb{A}_r$. In this article, we refine this dilation theorem and explicitly construct such a dilation. Let $K(\overline{\mathbb{A}}_r)$ be the smallest positive constant for which $\overline{\mathbb{A}}_r$ is a $K(\overline{\mathbb{A}}_r)$-spectral set for operators in $Q\mathbb{A}_r$. A significant result due to Tsikalas established the lower bound $K(\overline{\mathbb{A}}_r) \geq 2$, refining earlier estimates. Recently, Pascoe proved that $\displaystyle K(\overline{\mathbb{A}}_r)\leq  2\left(1+\frac{2r^2}{r^4-1}\right)$ and hence $K(\overline{\mathbb{A}}_r) \to 2$ as $r \to \infty$. In this article, two alternative proofs of Pascoe's upper bound are presented. The first one capitalizes a dilation theorem due to McCullough and Pascoe, while the second involves a certain variety in the Euclidean biball. In the multivariable setting, we show that the biannulus $\overline{\mathbb{A}}_r^2$ is a $K$-spectral set for some $K>0$ for commuting pairs of operators in $Q\mathbb{A}_r$. Furthermore, we derive upper and lower bounds on the smallest spectral constant $K$ for which certain classes of operator tuples in $Q\mathbb{A}_r$ have the closed polyannulus $\overline{\mathbb{A}}_r^n$ as a $K$-spectral set. If we denote the smallest constant by $K(\overline{\mathbb{A}}_r^2)$ for commuting pairs, and $K_{dc}(\overline{\mathbb{A}}_r^n)$ for doubly commuting $n$-tuples in $Q\mathbb{A}_r$, then the resulting bounds are given by
		\[
	2^n \leq K_{dc}(\overline{\mathbb{A}}_r^n) \leq \left(\frac{3r^2-1}{r^2-1}\right)^n \quad \text{and} \quad 2^2 \leq K(\overline{\mathbb{A}}_r^2) \leq \left[4+\left(\frac{r^2+1}{r^2-1}\right)^2+4\left(\frac{r^2+1}{r^2-1}\right)^{1\slash 2}\right],
		\]
	which further imply that $\displaystyle 2^n \leq \lim_{r \to \infty} K_{dc}(\overline{\mathbb{A}}_r^n) \leq 3^n$ and $\displaystyle 2^2 \leq \lim_{r \to \infty} K(\overline{\mathbb{A}}_r^2) \leq 3^2$.	
		\end{abstract}	
		
	\maketitle
	
\noindent 
\section{Introduction}\label{sec_intro}

\vspace{0.1cm}
 
\noindent Throughout the paper, all operators are bounded linear maps defined on complex Hilbert spaces. A contraction is an operator with norm at most $1$.  We shall use the following notations:  $\mathbb{C}$ denotes the complex plane, $\mathbb{D}$ is the open unit disc $\{z : |z|<1\}$ and $\mathbb{T}$ is the unit circle $\{z : |z|=1\}$. For a compact set $X \subseteq \mathbb{C}^n$, we denote by $\text{Rat}(X)$ the algebra of rational functions with singularities outside $X$, which is equipped with the supremum norm $\|.\|_{\infty, X}$. For $K>0$, we say that a compact set $X \subset \mathbb{C}^n$ is a \textit{$K$-spectral set} for a commuting operator tuple $\underline{T}=(T_1, \dotsc, T_n)$ if its Taylor joint spectrum $\sigma_T(\underline{T}) \subseteq X$ and $\|f(\underline{T})\| \leq K\|f\|_{\infty, X}$ for all $f \in \text{Rat}(X)$. Such a constant $K$ is called a \textit{spectral constant}. The set $X$ is called a \textit{spectral set} for $\underline T$ if it is a $K$-spectral set with $K=1$. For $r>1$, consider the annulus 
\[
\mathbb{A}_r=\{z \in \mathbb{C}: r^{-1}<|z|<r\}.
\] 
Let $Q\mathbb{A}_r$ be the class of invertible operators $T$ such that $\|T\|, \|T^{-1}\| \leq r$. The class $Q\A_r$ is known as the \textit{quantum annulus}. A central question to determine the smallest spectral constant $K(\CA_r)$ such that $\CA_r$ is a $K(\CA_r)$-spectral set for every operator in $Q\mathbb{A}_r$ still remains open. However, there are already a considerable number of research articles in this direction, e.g., see \cite{Badea, CrouI, CrouII, Pascoe, Paulsen, Shields, TsikalasII} and the references therein. A notable lower bound $K(\CA_r) \geq 2$ was established by Tsikalas \cite{TsikalasII}, improving the earlier bounds on $K(\CA_r)$ obtained in \cite{Badea}. Also, Crouzeix and Greenbaum \cite{CrouII} showed that $K(\CA_r) \leq 1+\sqrt{2}$. Later, Pascoe proved in \cite{Pascoe} that
\begin{align}\label{eqn_ub}
K(\CA_r) \leq 2\left(1 + \frac{2r^2}{r^4 - 1}\right).
\end{align}
Recently, Crouzeix \cite{CrouIII} obtained a sharper estimate on $K(\CA_r)$ for a quantum annulus operator $T$ acting on a finite dimensional Hilbert space, where $\|T\|, \|T^{-1}\| <r$. Combining the results due to Tsikalas \cite{TsikalasII} and Pascoe \cite{Pascoe}, it follows that ${\displaystyle \lim_{r \to \infty}K(\CA_r)=2}$. Pascoe obtained the bound as in \eqref{eqn_ub} by establishing a correspondence between $\A_r$ and the hyperbola $\mathbb{H}_r=\{(z, w) \in \mathbb{D}^2 : zw=1\slash r^2\}$ via the map $\varphi: \A_r \to \mathbb{H}_r$ given by $\varphi(z)=(z\slash r, z^{-1}\slash r)$. So, the problem of finding $K(\CA_r)$ for operators in $Q\A_r$ reduces to a similar problem in $\mathbb{H}_r$. In the hyperbola setting, the key idea in \cite{Pascoe} is to dilate operator pairs of the form $(Z, W)$, with $ZW=1\slash r^2$ and $\|Z\|, \|W\| < 1$, to commuting contractions $(\widehat{Z}, \widehat{W})$ such that $\widehat{Z}\widehat{W}=r^{-2}, U=r^2(1+r^2)^{-1}(\widehat{Z}+\widehat{W}^*)$ is a unitary and $\|f(Z, W)\|\leq \|f(\widehat{Z}, \widehat{W})\|$ for every holomorphic function $f$ on $\mathbb{H}_r$. Further, estimating $\|Uf(\widehat{Z}, \widehat{W})U^*\|$ gives a bound on $\|f(Z, W)\|$. Consequently, for every $T \in Q\A_r$ and $g \in \text{Rat}(\CA_r)$, it is proved in \cite{Pascoe} that
\begin{align}\label{eqn_01}
	\|g(T)\| \leq 2\left(1+\frac{2r^2}{r^4-1}\right)\|g\|_{\infty,\CA_r} \quad \text{and so,} \quad K(\CA_r) \leq 2\left(1+\frac{2r^2}{r^4-1}\right). 
\end{align}
In this article, we revisit Pascoe’s result and provide two independent proofs in Section \ref{sec_02}. The first one is based on a dilation-theoretic approach utilizing the dilation theorem for an operator in $Q\A_r$ due to McCullough and Pascoe \cite{Pas-McCull}. The second proof follows a different line, making use of a certain variety in the Euclidean biball as studied by Pal and Tomar in \cite{PalII}, offering a more geometric perspective. In particular, it is shown that Pascoe’s methods from \cite{Pascoe}, together with the dilation result from \cite{Pas-McCull} for quantum annulus, give the same estimate as in \eqref{eqn_01} without using the aforementioned dilation theorem in the hyperbola setting. In this direction, we recall an important result by McCullough and Pascoe \cite{Pas-McCull} providing the following characterization of operators in $Q\mathbb{A}_r$. For this purpose, given an invertible operator $T$, we set
\[
\beta(T^*, T):= (r^2+r^{-2})-T^*T-(T^*T)^{-1}.
\] 

\begin{prop}[\cite{Pas-McCull}, Proposition 3.2]\label{prop_3.2}
	An invertible operator $T \in Q\mathbb{A}_r$ if and only if $\beta(T^*, T)\geq 0$.
\end{prop}

We say that an invertible operator $T$ is a quantum annulus unitary if $\beta(T^*, T) = 0$. In Proposition \ref{prop_202}, we find a necessary and sufficient condition such that $A\otimes J+ B\otimes J^{-*}$ becomes an isometry, where $J$ is a quantum annulus unitary. This result is of independent interest. The significance of this class in the study of $Q\mathbb{A}_r$ becomes evident from the following result, showing that every operator in $Q\mathbb{A}_r$ dilates to a quantum annulus unitary.

\begin{thm}[\cite{Pas-McCull}, Theorem 1.1 \& Proposition 3.2]\label{thm_01}
	Let $T \in Q\A_r$ be an operator acting on a Hilbert space $\HS$. Then  there exists a quantum annulus unitary $J$ on a Hilbert space $\KS \supseteq \HS$ such that 
	$T^n=P_{\HS}J^n|_{\HS}$ for all $n \in \mathbb{Z}$.
\end{thm}

In Section \ref{sec_02}, we provide a refinement of the above dilation theorem by showing that every operator $T$ in $Q\mathbb{A}_r$ acting on a Hilbert space $\HS$ extends to a quantum annulus unitary, together with its explicit construction. The motivation of such a construction comes from an earlier work due to Pascoe \cite{Pascoe}. The dilation space we construct is $\mathcal{H} \oplus \mathcal{H}$, while the dilation obtained by the authors of \cite{Pas-McCull} acts on $L^2(\mathbb{T}) \otimes \mathcal{H}$ when $\sigma((T^*T)^{1/2})$ is a finite subset of $(r^{-1}, r)$ and in the general case, the dilation space arises via Stinespring’s dilation theorem. Having established the upper bound as in \eqref{eqn_ub} for operators in $Q\A_r$, we explore the multivariable case. Beyond the one-variable setting, much less is known yet about the $K$-spectral sets for commuting tuples of operators in $Q\A_r$. In Section \ref{sec_biannulus}, it is established that the biannulus $\CA_r \times \CA_r$ is a $K$-spectral for commuting pairs of operators in \(Q\mathbb{A}_r\), where
\[
K = 4 + \left(\frac{r^2 + 1}{r^2 - 1}\right)^2 + 4\left(\frac{r^2 + 1}{r^2 - 1}\right)^{1/2}.
\]
This estimate is derived via the techniques due to Shields \cite{Shields}. A key ingredient in the proof of this result is Ando's inequality \cite{Ando}, which states that for a commuting pair of contractions $(T_1, T_2)$,
$
\|p(T_1, T_2)\| \leq \|p\|_{\infty, \mathbb{T}^2}
$
for every polynomial $p \in \mathbb{C}[z_1, z_2]$. The same algorithm that is used for commuting pairs of operators in $Q\A_r$ does not extend to commuting $n$-tuples for $n \geq 3$. The underlying reason is the failure of Ando's inequality for a commuting $n$-tuple of contractions when $n \geq 3$, a fact established by Parrott in \cite{Par} via a counter example.

\smallskip 
 
Though Ando's type inequality fails for general commuting tuples of contractions, it holds for doubly commuting contractions (e.g., see \cite{NagyFoias6}). Recall that an operator tuple $(T_1, \dotsc, T_n)$ is said to be \textit{doubly commuting} if $T_iT_j=T_jT_i$ and $T_iT_j^*=T_j^*T_i$ for $1 \leq i, j \leq n$ with $i \ne j$. In Section \ref{sec_polyannulus}, we prove that if  $(T_1, \dotsc, T_n)$ is a doubly commuting tuple of operators in $Q\mathbb{A}_r$, then
\[
\|g(T_1, \dotsc, T_n)\| \leq \left(\frac{3r^2 - 1}{r^2 - 1}\right)^n \|g\|_{\infty, \CA_r^n}
\]
for all $g \in \text{Rat}(\CA_r^n)$. The above estimate is obtained by using classical Cauchy's estimates in the multivariable Laurent series representation of a holomorphic function on $\CA_r^n$. The proofs are technical and involve laborious computations.

\smallskip

After establishing that $\CA_r^2$ and $\CA_r^n$ are $K$-spectral sets for the commuting pair and doubly commuting $n$-tuple of operators in $Q\A_r$, respectively, the natural question is to determine the optimal constant in each case. While the general problem remains unsettled even for a single operator in $Q\A_r$, we establish in Section \ref{sec_lower} certain upper and lower bounds for the smallest such constants in each case. Let us denote by $K(\CA_r^2)$ the smallest constant for which every commuting pair in $Q\A_r$ admits $\CA_r^2$ as a $K(\CA_r^2)$-spectral set, and by $K_{dc}(\CA_r^n)$ the corresponding optimal constant for doubly commuting $n$-tuples of operators in $Q\A_r$. The bounds that we obtain are the following:
	\[
2^n \leq K_{dc}(\overline{\mathbb{A}}_r^n) \leq \left(\frac{3r^2-1}{r^2-1}\right)^n \quad \text{and} \quad 2^2 \leq K(\overline{\mathbb{A}}_r^2) \leq \left[4+\left(\frac{r^2+1}{r^2-1}\right)^2+4\left(\frac{r^2+1}{r^2-1}\right)^{1\slash 2}\right],
\]
which show that $\displaystyle 2^n \leq \lim_{r \to \infty} K_{dc}(\overline{\mathbb{A}}_r^n) \leq 3^n$ and $\displaystyle 2^2 \leq \lim_{r \to \infty} K(\overline{\mathbb{A}}_r^2) \leq 3^2$.

\vspace{0.1cm}

\section{Dilation and $K$-spectral set for operators in $Q\A_r$: An alternative approach}\label{sec_02}

\vspace{0.1cm}

\noindent As stated in Theorem \ref{thm_01}, McCullough and Pascoe \cite{Pas-McCull} proved that every operator $T$ in $Q\A_r$ acting on a Hilbert space $\HS$ dilates to a quantum annulus unitary $J$, i.e., 
\[
T^n=P_{\HS}J^n|_{\HS} \quad \text{for all $n \in \Z$}
\] 
with $\beta(J^*, J)=(r^2+r^{-2})-J^*J-(J^*J)^{-1}=0$. Their proof first considers the case when the spectrum of $(T^*T)^{1/2}$ is a finite subset of $(r^{-1}, r)$, establishing the dilation  of $T$ to a certain multiplication operator on $L^2(\mathbb{T}) \otimes \mathcal{H}$. The general case follows from Stinespring’s dilation theorem. In this section, an explicit construction of such a dilation on the space $\mathcal{H} \oplus \mathcal{H}$ is obtained. As an application, we obtain the upper bound due to Pascoe \cite{Pascoe} as in \eqref{eqn_ub} through two alternative and independent approaches. While the first one capitalizes the dilation theorem of operators in $Q\A_r$, the second one employs a certain variety in the biball studied by the authors of \cite{PalII}. To this end, we begin with the following refinement of the dilation result stated in Theorem \ref{thm_01}.

\begin{thm}\label{lem_01}
	Every operator $T \in Q\A_r$ extends $($and hence dilates$)$ to a quantum annulus unitary 
	\[
	\hat{T}=\begin{bmatrix}
		T & T(T^*T)^{-1\slash 2}\beta(T^*, T)^{1\slash 2}\\
		0 & T^{-*}
	\end{bmatrix}.
	\]
\end{thm}

\begin{proof}
	Evidently, $\beta(T^*, T)$ commutes with $T^*T$ and so, an application of spectral theorem gives the commutativity of $(T^*T)^{1\slash 2}$ and $\beta(T^*, T)^{1\slash 2}$. It is not difficult to see that 
	\[
	\hat{T}^{-1}=\begin{bmatrix}
		T^{-1} & -\beta(T^*, T)^{1\slash 2}(T^*T)^{-1\slash 2}T^*\\
		0 & T^*
	\end{bmatrix}. 
	\]
	Using the definition of $\beta(T^*, T)$, a few steps of routine computations give that 
	\begin{align*}
		\hat{T}^*\hat{T}
		&=\begin{bmatrix}
		T^*T & (T^*T)^{1\slash 2}\beta(T^*, T)^{1\slash 2}\\
		\beta(T^*, T)^{1\slash 2}(T^*T)^{1\slash 2} & (r^2+r^{-2})-T^*T 
		\end{bmatrix}  \ \ \text{and} \\
			(\hat{T}^*\hat{T})^{-1}
		&=\begin{bmatrix}
		(r^2+r^{-2})-T^*T & -\beta(T^*, T)^{1\slash 2}(T^*T)^{1\slash 2}\\
		-(T^*T)^{1\slash 2}\beta(T^*, T)^{1\slash 2} & T^*T 
		\end{bmatrix}.
	\end{align*}
Consequently, $\beta(\wT^*, \wT)=0$ and $T^n=\wT^n|_{\HS}$ for all $n \in \Z$. The proof is now complete.
\end{proof}

As an application, we provide an alternative proof to the following result due to Pascoe \cite{Pascoe}.

\begin{thm}\label{thm_main}
	Let $T \in Q\mathbb{A}_r$. Then $\displaystyle \|g(T)\| \leq 2\left(1+\frac{2r^2}{r^4-1}\right)\|g\|_{\infty,\CA_r}$ for all $g \in \text{Rat}(\CA_r)$.
\end{thm}

\begin{proof} Let $T \in Q\A_r$ be an operator on a Hilbert space $\HS$. By Theorem \ref{thm_01}, there is an operator $J \in Q\A_r$ on a Hilbert space $\KS \supseteq \HS$ such that $\beta(J^*, J)=0$ and $g(T)=P_{\HS}g(J)|_{\HS}$ for all $g$ in $\text{Rat}(\CA_r)$. Then $U=r(1+r^2)^{-1}(J+J^{-*})$ is a unitary on $\KS$ as $\beta(J^*, J)=0=J\beta(J^*, J)J^{-1}$. Let $g \in \text{Rat}(\CA_r)$. We write $\displaystyle g(z)=a_0+zg^+(z)+(1\slash z)g^-\left(1\slash z\right)$, where $a_0+zg^+(z)$ and $(1\slash z)g^-\left(1\slash z\right)$ are the analytic  and principal parts in the Laurent series expansion of $g$ respectively. Then
\begin{align*}
	Ug(J)U^*
	&=U\left(a_0+Jg^+(J)+J^{-1}g^-(J^{-1})\right)U^*\\
	&=\frac{r}{1+r^2}\left[U\left(Jg^+(J)+a_0\right)J^*+J^{-*}\left(J^{-1}g^{-}(J^{-1})+a_0\right)U^*+Ug^+(J)+g^-(J^{-1})U^*  \right].
	\end{align*}  	
Clearly, $\|g(T)\| \leq \|g(J)\|=\|Ug(J)U^*\|$. Applying the Cauchy estimates as in Theorem 2 of \cite{Pascoe}, and using the facts that $U$ is a unitary and $\|J\|, \|J^{-1}\| \leq r$, we have that
\begin{align*}
\|g(T)\|
&\leq \|Ug(J)U^*\|\\
&\leq \frac{r}{1+r^2}\bigg[r\|rzg^+(rz)+a_0\|_{\infty, \mathbb{T}}+r\|rwg^-(rw)+a_0\|_{\infty, \mathbb{T}}+\|g^+(rz)\|_{\infty, \mathbb{T}}+\|g^-(rw)\|_{\infty, \mathbb{T}} \bigg]\\	
& \leq \frac{r}{1+r^2}\bigg[r\frac{r^2}{r^2-1}+r\frac{r^2}{r^2-1}+\frac{2r^2-1}{r(r^2-1)}+\frac{2r^2-1}{r(r^2-1)} \bigg]\|g\|_{\infty, \CA_r} \ \ [\text{by Cauchy's estimate}]\\
&=2\left(1+\frac{2r^2}{r^4-1}\right)\|g\|_{\infty,\CA_r}.
\end{align*}
The proof is now complete. 
\end{proof}

The authors of \cite{PalII} considered the variety given by intersection of the zero set $Z(q_0)$ of the polynomial $q_0$ with biball $\mathbb{B}_2$, where $q_0(z, w)=zw-(r^{2}+r^{-2})^{-1}$. They proved that an invertible operator $T$ has $\CA_r$ as a spectral set if and only if the operator pair 
\[
\left(\frac{T}{\sqrt{r^2+r^{-2}}}, \frac{T^{-1}}{\sqrt{r^2+r^{-2}}}\right)
\]
has $Z(q_0) \cap \overline{\mathbb{B}}_2$ as a spectral set, which makes a key connection between the non-polynomially convex set $\CA_r$ and the polynomially convex set $Z(q_0) \cap \overline{\mathbb{B}}_2$.  An interested reader is referred to Section 9 in \cite{PalII} for further discussion on the interaction of this variety with the operators associated with $\A_r$. We now present another proof of Theorem \ref{thm_main} capitalizing the variety $Z(q_0) \cap \overline{\mathbb{B}}_2$.

\begin{proof} 
	Let $T \in Q\A_r$ be an operator acting on a Hilbert space $\HS$. We have by Proposition \ref{prop_3.2} that $\beta(T^*, T)=(r^2+r^{-2})-T^*T-(T^*T)^{-1} \geq 0$. Define 
	\[
	(\hat{A}, \hat{B})=\left(\frac{\hat{T}}{\sqrt{r^2+r^{-2}}}, \frac{\hat{T}^{-1}}{\sqrt{r^2+r^{-2}}}\right), \quad \text{where} \quad \hat{T}=\begin{bmatrix}
		T & T(T^*T)^{-1\slash 2}\beta(T^*, T)^{1\slash 2}\\
		0 & T^{-*}
	\end{bmatrix} .
	\]
It follows from the previous proof of Theorem \ref{lem_01} that $\beta(\wT^*, \wT) =0$ and by Proposition \ref{prop_3.2}, $\|\wT\|, \|\wT^{-1}\| \leq r$. Thus, $\sigma(\wT) \subseteq \CA_r$ and by spectral mapping theorem, $\sigma_T(\wA, \wB) \subseteq Z(q_0)\cap \overline{\mathbb{B}}_2$. It is not difficult to see that $\displaystyle U=\frac{\sqrt{r^4+1}}{r^2+1}(\wA+\wB^*)$ is a unitary on $\HS \oplus \HS$. Let $f(z, w)$ be a polynomial on the variety $Z(q_0) \cap \CB_2$. We can write $f(z, w)=a_0+zf^+(z)+wf^-(w)$ on $Z(q_0) \cap \CB_2$, where $f^+(z)$ and $f^-(w)$ are holomorphic polynomials in $z$ and $w$ respectively. A routine calculation shows that
\begin{align*}
Uf(\wA, \wB)U^*&=U\left(a_0+\wA f^+(\wA)+\wB f^-(\wB) \right)U^*\\
&=\frac{\sqrt{r^4+1}}{r^2+1}\left[U(\wA f^+(\wA)+a_0)\wA^*+\wB^*(\wB f^-(\wB)+a_0)U^*+\frac{Uf^+(\wA)}{(r^2+r^{-2})}+\frac{f^-(\wB)U^*}{(r^2+r^{-2})}  \right].
\end{align*}
Clearly, $f(\wA, \wB)$ is an extension of $f(A, B)$ and so, $\|f(A, B)\| \leq \|f(\wA, \wB)\|=\|Uf(\wA, \wB)U^*\|$. For the sake of brevity, write $a_r=1\slash\sqrt{r^2+r^{-2}}$. Since $\displaystyle \|\wA\|, \|\wB\| \leq ra_r$ and $U$ is a unitary, we have
\begin{align*}
	 \|f(A, B)\|& \leq \|Uf(\wA, \wB)U^*\|\\
	&\leq \frac{\sqrt{r^4+1}}{r^2+1}\bigg[ra_r\| ra_rz f^+(ra_rz)+a_0\|_{\infty, \mathbb T}+
	ra_r\| ra_rw \ f^-(ra_rw)+a_0\|_{\infty, \mathbb T} \\ 
	& \qquad \qquad \qquad +a_r^2\| f^+(ra_rz)\|_{\infty, \mathbb T}+ a_r^2 \|f^-(ra_rw)\|_{\infty, \mathbb T} \bigg]\\
	&\leq \frac{\sqrt{r^4+1}}{r^2+1}\bigg[ra_r\frac{r^2}{r^2-1}+
	ra_r\frac{r^2}{r^2-1}  +a_r^2\frac{2r^2-1}{ra_r(r^2-1)}+ a_r^2\frac{2r^2-1}{ra_r(r^2-1)} \bigg] \|f\|_{\infty, Z(q_0)\cap\CB_2}
	\\ &  \hspace{11cm} [\text{by Cauchy's estimate}]\\
	&=2\left(1+\frac{2r^2}{r^4-1}\right)\|f\|_{\infty, Z(q_0)\cap\CB_2}.
\end{align*} 
Since $Z(q_0) \cap \CB_2$ is polynomially convex, polynomials are dense in $\text{Rat}(Z(q_0) \cap \CB_2)$ and so,
\[
\|f(A, B)\| \leq 2\left(1+\frac{2r^2}{r^4-1}\right)\|f\|_{\infty, Z(q_0)\cap\CB_2}
\] 
for all $f \in \text{Rat}(Z(q_0) \cap \CB_2)$. For any $g \in \text{Rat}(\CA_r)$, we can write $g(z)=f(a_rz, a_rz^{-1})$ for some $f \in \text{Rat}(Z(q_0) \cap \CB_2)$ and thus, it follows that
\[
\|g(T)\| \leq \|g(\wT)\|=\|f(\wA, \wB)\|  \leq 2\left(1+\frac{2r^2}{r^4-1}\right)\|f\|_{\infty, Z(q_0)\cap\CB_2}\leq 2\left(1+\frac{2r^2}{r^4-1}\right)\|g\|_{\infty, \CA_r}.
\]
The proof is now complete. 
\end{proof} 
 	
For any operator $J$ with $\beta(J^*, J)=0$, we observed in the first proof of Theorem \ref{thm_main} that the operator $U=r(1+r^2)^{-1}(J+J^{-*})$ defines a unitary. In fact, we can replace $J$ by $J_0=J \oplus rI \oplus r^{-1}I$ and the same proof works. Since $\|J\|, \|J^{-1}\| \leq r$, it follows that $\|J_0\|=\|J_0^{-1}\|=r$. We ask a natural question here: What are all the scalars $a, b \in \mathbb C$ such that $a J+b J^{-*}$ is an isometry for a quantum annulus unitary $J$ with $\|J\|=\|J^{-1}\|=r$? In general, characterizing all operators $A, B$ such that $A\otimes J+ B\otimes J^{-*}$ is an isometry\slash unitary seems a problem of independent interest.

\begin{prop}\label{prop_202}
	Let $J$ be a quantum annulus unitary on a Hilbert space $\HS$ with $\|J\|=\|J^{-1}\|=r$ and let $(A, B)$ be operators (not necessarily commuting) acting on a Hilbert space $\mathcal{L}$. Then $A\otimes J+ B\otimes J^{-*}$ is an isometry if and only if $A^*A=B^*B$ and $A^*B+B^*A+(r^2+r^{-2})A^*A=I$. 
\end{prop} 	
 	
\begin{proof}
	Let $C=A\otimes J+ B\otimes J^{-*}$ and $c_r=r^{2}+r^{-2}$. A few steps of routine computations give that
	\begin{align}\label{eqn_201}
		C^*C&=\left[A^*\otimes J^*+ B^*\otimes J^{-1}\right]\left[A\otimes J+ B\otimes J^{-*}\right] \notag \\
		&=A^*A \otimes J^*J+ (A^*B+B^*A) \otimes I +B^*B \otimes J^{-1}J^{-*} \notag  \\
		&=(A^*A-B^*B) \otimes J^*J +(A^*B+B^*A+c_rB^*B)\otimes I ,
		\end{align}  
		where in the last equality, we have used the fact that $\beta(J^*, J)=0$. We also have that
	\begin{align}\label{eqn_202}
	C^*C&=A^*A \otimes J^*J+ (A^*B+B^*A) \otimes I +B^*B \otimes J^{-1}J^{-*} \notag  \\
&=A^*A \otimes (c_r-J^{-1}J^{-*})+ (A^*B+B^*A) \otimes I +B^*B \otimes J^{-1}J^{-*} \qquad [\text{as} \ \beta(J^*, J)=0 ] \notag  \\
	&=(B^*B-A^*A) \otimes J^{-1}J^{-*} +(A^*B+B^*A+c_rA^*A)\otimes I.
\end{align}  
Let $C$ be an isometry. Since $\|J\|=\|J^{-*}\|=r$, one can find sequences $\{y_n\}, \{z_n\}$ of elements in $\mathcal{H}$ such that $\|y_n\|=\|z_n\|=1$ and $\lim_{n \to \infty}\|Jy_n\|=\lim_{n \to \infty}\|J^{-*}z_n\|=r$. For $x \in \mathcal{L}$, we have that
\begin{align*}
\|C(x \otimes y_n)\|^2&=\la C^*C(x \otimes y_n), (x \otimes y_n) \ra\\
&=\la((A^*A-B^*B) \otimes J^*J +(A^*B+B^*A+c_rB^*B)\otimes I)(x \otimes y_n), (x \otimes y_n) \ra \qquad [\text{by} \ \eqref{eqn_201}]\\
&=\la (A^*A-B^*B)x, x \ra\|Jy_n\|^2+\la (A^*B+B^*A+c_rB^*B)x, x \ra
\end{align*} 
for every $n \in \mathbb{N}$. Using the fact that $C$ is an isometry and letting $n\to \infty$, it follows that 
\begin{align*}
\|x\|^2=\la (A^*A-B^*B)x, x \ra r^2+\la (A^*B+B^*A+c_rB^*B)x, x \ra
=\la (r^2A^*A+r^{-2}B^*B+A^*B+B^*A)x, x \ra 
\end{align*}
for all $x \in \mathcal{L}$ and so, 
\begin{align}\label{eqn_203}
	r^2A^*A+r^{-2}B^*B+A^*B+B^*A=I.
\end{align}
Using similar computations as above, we have that
\begin{align*} 
	\|C(x \otimes z_n)\|^2&=\la C^*C(x \otimes z_n), (x \otimes z_n) \ra\\
	&=\la((B^*B-A^*A) \otimes J^{-1}J^{-*} +(A^*B+B^*A+c_rA^*A)\otimes I)(x \otimes z_n), (x \otimes z_n) \ra \quad [\text{by} \ \eqref{eqn_202}]\\
	&=\la (B^*B-A^*A)x, x \ra\|J^{-*}z_n\|^2+\la (A^*B+B^*A+c_rA^*A)x, x \ra  
\end{align*}
for every $n \in \mathbb{N}$. Letting $n\to \infty$, it follows that 
\begin{align*}
	\|x\|^2=\la (B^*B-A^*A)x, x \ra r^2+\la (A^*B+B^*A+c_rA^*A)x, x \ra
	=\la (r^{-2}A^*A+r^{2}B^*B+A^*B+B^*A)x, x \ra 
\end{align*}
for all $x \in \mathcal{L}$  and so, 
\begin{align}\label{eqn_204}
r^{-2}A^*A+r^{2}B^*B+A^*B+B^*A=I.
\end{align}
It now follows from \eqref{eqn_203} and \eqref{eqn_204} that $A^*B+B^*A+c_rB^*B=I=A^*B+B^*A+c_rA^*A$ and so, $A^*A=B^*B$. The converse follows trivially from \eqref{eqn_201} which completes the proof.
\end{proof}
 	
The following result is an immediate corollary to the above result.
\begin{cor}
		Let $J$ be a quantum annulus unitary on a Hilbert space $\HS$ with $\|J\|=\|J^{-1}\|=r$ and let $(A, B)$ be operators (not necessarily commuting) acting on a Hilbert space $\mathcal{L}$. Then $A\otimes J+ B\otimes J^{-*}$ is a unitary if and only if 
		\begin{align*}
		& \ \ (i) \ A^*A=B^*B \quad (ii) \ A^*B+B^*A+(r^2+r^{-2})A^*A=I \\
		 & (iii) \ AA^*=BB^* \quad  (iv) \ AB^*+BA^*+(r^2+r^{-2})AA^*=I.
		\end{align*}
\end{cor}
 	 	
 \begin{proof}
 	Let $C=A\otimes J+ B\otimes J^{-*}$. The desired conclusion follows by applying Proposition \ref{prop_202} to both $C$ and its adjoint $C^*=A^*\otimes J^*+ B^*\otimes J^{-1}$.
 \end{proof}	
 	
The scalar case corresponding to the above corollary now becomes trivial. Let $J$ be a quantum annulus unitary on a Hilbert space $\HS$ with $\|J\|=\|J^{-1}\|=r$.  A simple application of the above corollary yields that $a J+b J^{-*}$ is a unitary if and only if $|a|=|b|$ and $|ra+r^{-1}b|=1$, which is equivalent to saying that  $b=a e^{i\theta}$ and $|a|^2=(r^2+r^{-2}+2\cos \theta)^{-1}$ for some $\theta \in \mathbb{R}$.

\medskip 

As mentioned earlier in connection with the operators in $Q\mathbb{A}_r$, Pascoe \cite{Pascoe} studied the variety in bidisc $\mathbb D^2$ given by
\[
\mathbb{H}_r=\{(z, w) \in \mathbb{D}^2 : zw=1\slash r^2\},
\]
which is called the quantum conservative hyperbola. The map $\varphi: \A_r \to \mathbb{H}_r$ given by $\varphi(z)=(z\slash r, z^{-1}\slash r)$ is a bijective holomorphic map with inverse $\varphi^{-1}(z, w)=rz$. It follows from maximum modulus principle that the distinguished boundary of $\A_r$ is its topological boundary $\partial \A_r$. Recall that the distinguished boundary of a compact set $X$ is the Shilov boundary of $X$ with respect to $Rat(K)$, i.e., the smallest closed subset $bX$ of $X$ such that every $f \in \text{Rat}(X)$ attains its modulus on $bX$. A point $w \in X$ is a \textit{peak point} if there exists $f \in \text{Rat}(X)$ such that $f(w)=1$ and $|f(z)|<1$ for all $z \ne w$ in $X$. The function $f$ is called a \textit{peaking function} for $w$. We conclude this section with a characterization of the distinguished boundary $b\overline{\mathbb{H}}_r$ of Pascoe's hyperbola $\mathbb{H}_r$ as in \cite{Pascoe}.

\begin{prop}
	$b\overline{\mathbb{H}}_r=\{(r^{-1}\alpha, r^{-1}\alpha^{-1}): \alpha \in \partial \A_r\}$.
\end{prop}

\begin{proof}
	Let $f: \overline{\mathbb{H}}_r \to \mathbb C$ be a continuous function that is holomorphic on $\mathbb{H}_r$. For $\varphi: \A_r \to \mathbb{H}_r$ given by $\varphi(z)=(z\slash r, z^{-1}\slash r)$, the map $g=f \circ \varphi: \CA_r \to \mathbb{C}$ is a continuous function that is holomorphic on $\A_r$. By maximum modulus principle, there exists $\alpha \in \partial \A_r$ such that $\|g\|_{\infty, \CA_r}=|g(\alpha)|$. It is not difficult to see that $\|f\|_{\infty, \overline{\mathbb{H}}_r}=|f(\varphi(\alpha))|$. Consequently, any function in $\text{Rat}(\overline{\mathbb{H}}_r)$ attains its maximum modulus at some point in the closed set 
	\[
	\varphi(\partial \A_r)=\{(r^{-1}\alpha, r^{-1}\alpha^{-1}): \alpha \in \partial \A_r\}
	\]
and so, $b	\overline{\mathbb{H}}_r\subseteq \varphi(\partial \A_r)$. Take any point $\varphi(\alpha) \in \varphi(\partial\A_r)$. Since any point in $\partial \A_r$ is a peak point, there exists $g \in \text{Rat}(\CA_r)$ such that $g(\alpha)=1$ and $|g(\beta)|<1$ for all $\beta \in \CA_r$ with $\beta \ne \alpha$. Consequently, the map $f=g\circ \varphi^{-1} \in \text{Rat}(\overline{\mathbb{H}}_r)$ and $f$ is a peaking function for $\varphi(\alpha)$. Therefore,  $\phi(\alpha) \in b\overline{\mathbb{H}}_r$ and so, $\varphi(\partial \A_r) \subseteq b\overline{\mathbb{H}}_r$, which completes the proof.
\end{proof}

\section{Biannulus as a $K$-spectral set for a commuting pair of operators in $Q\A_r$}\label{sec_biannulus}

\noindent In Section \ref{sec_02}, it is proved that $\CA_r$ is a $K$-spectral set for operators in $Q\A_r$. However, beyond the one-variable setting, not much is known about $K$-spectral sets for commuting tuples of operators in $Q\A_r$. In this section, we show that the biannulus $\CA_r \times \CA_r$ is a $K$-spectral set for commuting pairs of operators in $Q\A_r$, where
\[
K=4+\left(\frac{r^2+1}{r^2-1}\right)^2+4\left(\frac{r^2+1}{r^2-1}\right)^{1\slash 2}.
\] 
The key ingredient in the proof of this result is the dilation theorem due to Ando \cite{Ando}.  For a commuting pair of contractions $(T_1, T_2)$ acting on a Hilbert space $\HS$, a fundamental theorem due to Ando \cite{Ando} guarantees the existence of a commuting pair of unitaries $(U_1, U_2)$ on a Hilbert space $\mathcal{K}$ containing $\HS$ such that 
$
g(T_1, T_2)=P_{\HS}g(U_1, U_2)|_{\HS}
$
for all $g \in \text{Rat}(\overline{\mathbb{D}}^2)$. Now, a simple application of spectral mapping principle gives that
\[
\|g(T_1, T_2)\| \leq \|g(U_1, U_2)\| =\|g\|_{\infty, \sigma_T(U_1, U_2)} \leq \|g\|_{\infty, \mathbb{T}^2}
\]
for all $g \in \text{Rat}(\overline{\mathbb{D}}^2)$. The inequality above is known as a von Neumann's type inequality for commuting pairs of contractions, commonly referred to as Ando's inequality. As an application, we obtain the following estimates for holomorphic functions on $\CA_r^2$.

\begin{prop}\label{prop_estimate}
	Let $g(z_1, z_2)$ be a non-zero holomorphic function on the biannulus $\CA_r^2$ and let the unique Laurent representation of $g$ on $\CA_r^2$ be given by
	\begin{equation*}
		g(z_1, z_2)=\overset{\infty}{\underset{n=-\infty}{\sum}}\ \overset{\infty}{\underset{m=-\infty}{\sum}}a_{n,m}\ z_1^n z_2^m=g_1(z_1, z_2)+g_2(z_1, z_2^{-1})+g_3(z_1^{-1}, z_2)+g_4(z_1^{-1}, z_2^{-1}),
	\end{equation*}
	where
	\begin{align*}
		& g_1(z_1, z_2)=\overset{\infty}{\underset{n=0}{\sum}}\ \overset{\infty}{\underset{m=0}{\sum}}a_{n,m} \ z_1^nz_2^m, \qquad \qquad  \qquad g_2(z_1, w_2)=\overset{\infty}{\underset{n=0}{\sum}}\ \overset{\infty}{\underset{m=1}{\sum}}a_{n,-m}z_1^nw_2^m, \\
		& g_3(w_1, z_2)=\overset{\infty}{\underset{n=1}{\sum}}\ \overset{\infty}{\underset{m=0}{\sum}}a_{-n,m} w_1^nz_2^m \qquad \text{and} \qquad g_4(w_1, w_2)=\overset{\infty}{\underset{n=1}{\sum}}\ \overset{\infty}{\underset{m=1}{\sum}}a_{-n,-m}w_1^nw_2^m.
	\end{align*}
	Then 
	\begin{align*}
		\frac{\|g_1\|_{\infty, r\mathbb{T} \times r\mathbb{T}}}{\|g\|_{\infty, \CA_r^2}} & \leq  \left[1+ \left(\frac{2}{\sqrt{r^4-1}}\right)+\left(\frac{1}{r^2-1}\right)^2\right], \ \
		\frac{\|g_2\|_{\infty, r\mathbb{T} \times r\mathbb{T}}}{\|g\|_{\infty, \CA_r^2}} \leq \left[1+\frac{1+r^2}{\sqrt{r^4-1}}+\frac{r^2}{(r^2-1)^2}\right], \\
		\frac{\|g_3\|_{\infty, r\mathbb{T} \times r\mathbb{T}}}{\|g\|_{\infty, \CA_r^2}}& \leq \left[1+\frac{1+r^2}{\sqrt{r^4-1}}+\frac{r^2}{(r^2-1)^2}\right] \ \ \ \text{and} 
		\ \ \ \frac{\|g_4\|_{\infty, r\mathbb{T} \times r\mathbb{T}}}{\|g\|_{\infty, \CA_r^2}} \leq \left[1+\frac{2r^2}{\sqrt{r^4-1}}+\left(\frac{r^2}{r^2-1}\right)^2 \ \right].
	\end{align*}	
\end{prop}
\begin{proof} 
Let $g$ be a holomorphic function on $\CA_r^2$ with the Laurent series representation given by
\begin{equation}\label{eqn_0301}
	g(z_1, z_2)=\overset{\infty}{\underset{n=-\infty}{\sum}}\ \overset{\infty}{\underset{m=-\infty}{\sum}}a_{n,m} \ z_1^nz_2^m \qquad \text{for} \  (z_1, z_2) \in \CA_r^2.
\end{equation}
The above series converges uniformly and absolutely on $\CA_r^2$ with the coefficients $a_{n,m}$ given by
\[
a_{n, m}=\frac{1}{(2\pi i)^2}\underset{|\zeta_1|=r_1}{\int} \ \underset{|\zeta_2|=r_2}{\int}\frac{g(\zeta_1, \zeta_2)}{\zeta_1^{n+1}\zeta_2^{m+1}}d\zeta_1 d\zeta_2 \quad (n, m \in \mathbb{Z}),
\]
where $r_1, r_2$ are scalars with $1\slash r \leq r_1, r_2 \leq r$. By Cauchy's estimate, $\displaystyle |a_{n, m}| \leq \frac{\|g\|_{\infty, \CA_r^2}}{r^{|n|+|m|}}$ for every $n, m \in \mathbb{Z}$. An interested reader is referred to Chapter II in \cite{Range} for further details. Then
\begin{align*}
	g(z_1, z_2)
	&=g_1(z_1, z_2)+g_2(z_1, z_2^{-1})+g_3(z_1^{-1}, z_2)+g_4(z_1^{-1}, z_2^{-1}),
\end{align*}
where 
\begin{align}\label{eqn_g1234}
	& g_1(z_1, z_2)=\overset{\infty}{\underset{n=0}{\sum}}\ \overset{\infty}{\underset{m=0}{\sum}}a_{n,m} \ z_1^nz_2^m, \quad \qquad g_2(z_1, w_2)=\overset{\infty}{\underset{n=0}{\sum}}\ \overset{\infty}{\underset{m=1}{\sum}}a_{n,-m}z_1^nw_2^m, \notag \\
	& g_3(w_1, z_2)=\overset{\infty}{\underset{n=1}{\sum}}\ \overset{\infty}{\underset{m=0}{\sum}}a_{-n,m} w_1^nz_2^m, \qquad g_4(w_1, w_2)=\overset{\infty}{\underset{n=1}{\sum}}\ \overset{\infty}{\underset{m=1}{\sum}}a_{-n,-m}w_1^nw_2^m.
\end{align}
We compute the estimates for $\|g_j\|_{\infty, r\mathbb{T} \times r\mathbb{T}}$ for $1 \leq j \leq 4$. The approach adopted here is motivated by the methods outlined in Section 6 of \cite{Shields}, where the author of \cite{Shields} studied the analytic structure of holomorphic functions on an annulus in the complex plane. From here onwards, the proof is divided into four steps for the better understanding of the readers.

\medskip 

\noindent \textbf{(1) An estimate for $\|g_1\|_{\infty, r\mathbb{T} \times r\mathbb{T}}$}. Let $(\xi_1, \xi_2)=(re^{i\theta}, re^{i\phi})$ for some $\theta, \phi \in \mathbb{R}$. Define
\[
b_{-n}=\overset{\infty}{\underset{m=-\infty}{\sum}}a_{-n,m} \xi_2^m \quad \text{and} \quad c_{-m}=\overset{\infty}{\underset{n=-\infty}{\sum}}a_{n, -m} \xi_1^n \quad (n, m \in \mathbb{Z}).
\] 
The Laurent series representations of the functions $g(., \xi_2)$ and $g(\xi_1, .)$ in $\CA_r$ are given by
\begin{align}\label{eqn_0302}
	g(z_1, \xi_2)=\overset{\infty}{\underset{n=-\infty}{\sum}}b_{n} z_1^n \qquad \text{and} \qquad 	g(\xi_1, z_2)=\overset{\infty}{\underset{m=-\infty}{\sum}}c_m z_2^m,
\end{align}
respectively. By Cauchy's estimate in one-variable, we have that
\begin{align}\label{eqn_0303}
|b_{-n}|= \left|  \overset{\infty}{\underset{m=-\infty}{\sum}}a_{-n,m} \xi_2^m\right| \leq \frac{\|g\|_{\infty, \CA_r^2}}{r^{|n|}} \quad \text{and} \quad |c_{-m}|=\left|  \overset{\infty}{\underset{n=-\infty}{\sum}}a_{n,-m} \xi_1^n\right| \leq \frac{\|g\|_{\infty, \CA_r^2}}{r^{|m|}}
\end{align}
for $n, m \in \mathbb{Z}$. It is not difficult to see that 
\begin{align*}
	\frac{1}{2\pi} \ \overset{2\pi}{\underset{0}{\int}}\left|g\left(\frac{1}{re^{i\zeta}}, \xi_2\right)\right|^2 d \zeta=\overset{\infty}{\underset{n=-\infty}{\sum}}|b_{-n}|^2r^{2n} \quad \text{and} \quad 	\frac{1}{2\pi} \ \overset{2\pi}{\underset{0}{\int}}\left|g\left(\xi_1, \frac{1}{re^{i\zeta}}\right)\right|^2 d\zeta=\overset{\infty}{\underset{m=-\infty}{\sum}}|c_{-m}|^2r^{2m}.
\end{align*}
Consequently, we have that
\begin{align}\label{eqn_03004}
	\left(\overset{\infty}{\underset{n=-\infty}{\sum}}|b_{-n}|^2r^{2n}\right)^{1\slash 2} \leq \|g\|_{\infty, \CA_r^2} \quad \text{and} \quad \left(\overset{\infty}{\underset{m=-\infty}{\sum}}|c_{-m}|^2r^{2m}\right)^{1\slash 2} \leq \|g\|_{\infty, \CA_r^2}.
\end{align}
Note that
\begin{align}\label{eqn_0304}
	\left|\overset{\infty}{\underset{n=0}{\sum}} \ \overset{\infty}{\underset{m=-\infty}{\sum}}a_{n,m} \xi_1^n \xi_2^m\right| = \left|g(\xi_1, \xi_2)-\overset{\infty}{\underset{n=1}{\sum}} \ \overset{\infty}{\underset{m=-\infty}{\sum}}a_{-n,m} \frac{\xi_2^m}{\xi_1^n}\right| \notag 
	&=\left|g(\xi_1, \xi_2)-\overset{\infty}{\underset{n=1}{\sum}} \  \frac{b_{-n}}{\xi_1^n}\right| \notag \\
	 &\leq \|g\|_{\infty, \CA_r^2}+\overset{\infty}{\underset{n=1}{\sum}}\frac{|b_{-n}|}{r^{n}} \notag \\	
	 & \leq \|g\|_{\infty, \CA_r^2}+ \sqrt{\left(\overset{\infty}{\underset{n=1}{\sum}}|b_{-n}|^2r^{2n}\right)\left(\overset{\infty}{\underset{n=1}{\sum}}\frac{1}{r^{4n}}\right)} \notag \\
	 & \leq \|g\|_{\infty, \CA_r^2}\left[1+ \left(\frac{1}{r^4-1}\right)^{1\slash 2}\right],
\end{align}
where the last inequality follows from \eqref{eqn_03004}. A few steps of simple calculations give that
\begin{align}\label{eqn_0307}
	\left|	\overset{\infty}{\underset{n=0}{\sum}} \ \overset{\infty}{\underset{m=1}{\sum}}\frac{a_{n,-m}}{\xi_2^m}\xi_1^n
	\right|
	&=\left|\overset{\infty}{\underset{n=-\infty}{\sum}} \ \overset{\infty}{\underset{m=1}{\sum}}\frac{a_{n,-m}}{\xi_2^m}\xi_1^n-\overset{-1}{\underset{n=-\infty}{\sum}} \ \overset{\infty}{\underset{m=1}{\sum}}\frac{a_{n,-m}}{\xi_2^m}\xi_1^n
	\right| \notag \\
	&\leq	\left|\overset{\infty}{\underset{m=1}{\sum}} \frac{c_{-m}}{\xi_2^m}\right|+\left|\overset{\infty}{\underset{n=1}{\sum}} \ \overset{\infty}{\underset{m=1}{\sum}}\frac{a_{-n,-m}}{\xi_1^n\xi_2^m}
	\right| \notag \\
	& \leq \overset{\infty}{\underset{m=1}{\sum}} \frac{|c_{-m}|}{r^m}+\overset{\infty}{\underset{n=1}{\sum}} \ \overset{\infty}{\underset{m=1}{\sum}}\frac{|a_{-n,-m}|}{r^n r^m}
 \notag \\
	& \leq 	\sqrt{\left(\overset{\infty}{\underset{m=1}{\sum}}|c_{-m}|^2r^{2m}\right)\left(\overset{\infty}{\underset{m=1}{\sum}}\frac{1}{r^{4m}}\right)} +\overset{\infty}{\underset{n=1}{\sum}} \ \overset{\infty}{\underset{m=1}{\sum}}\frac{\|g\|_{\infty, \CA_r^2}}{r^{2n}r^{2m}}  \qquad \left[\text{as} \ |a_{n,m}| \leq \frac{\|g\|_{\infty, \CA_r^2}}{r^{|n|+|m|}}\right] \notag \\
	&=\|g\|_{\infty, \CA_r^2}\left[\left(\frac{1}{r^4-1}\right)^{1\slash 2}+\left(\frac{1}{r^2-1}\right)^2\right],
\end{align}
where the last inequality follows from \eqref{eqn_03004}. Finally, we have by \eqref{eqn_0304} and \eqref{eqn_0307} that
\begin{align*}
	|g_1(\xi_1, \xi_2)|
	&=\left| \overset{\infty}{\underset{n=0}{\sum}}\ \overset{\infty}{\underset{m=-\infty}{\sum}}a_{n,m} \xi_1^n\xi_2^m-\overset{\infty}{\underset{n=0}{\sum}}  \ \overset{\infty}{\underset{m=1}{\sum}}\frac{a_{n,-m}}{\xi_2^m} \xi_1^n\right| 
	 \leq \|g\|_{\infty, \CA_r^2}\left[1+ \left(\frac{2}{\sqrt{r^4-1}}\right)+\left(\frac{1}{r^2-1}\right)^2\right] 
\end{align*}
and so, 
\begin{align}\label{eqn_0308}
	\|g_1\|_{\infty, r\mathbb{T} \times r\mathbb{T}} \leq \|g\|_{\infty, \CA_r^2} \left[1+ \left(\frac{2}{\sqrt{r^4-1}}\right)+\left(\frac{1}{r^2-1}\right)^2\right].
\end{align}

\medskip 

\noindent \textbf{(2) An estimate for $\|g_2\|_{\infty, r\mathbb{T} \times r\mathbb{T}}$}. Let $(\xi_1, \xi_2)=(re^{i\theta}, re^{i\phi})$ for some $\theta, \phi \in \mathbb{R}$. Define
\[
d_{-n}=\overset{\infty}{\underset{m=-\infty}{\sum}}\frac{a_{-n,m}}{\xi_2^m} \quad \text{and} \quad e_{-m}=\overset{\infty}{\underset{n=-\infty}{\sum}}a_{n, m} \xi_1^n \quad (n, m \in \mathbb{Z}).
\]  
The Laurent series representations of $g(., \xi_2^{-1})$ and $g(\xi_1, .)$ in $\CA_r$ are given by
\begin{align}\label{eqn_0309}
	g(z_1, \xi_2^{-1})
	=\overset{\infty}{\underset{n=-\infty}{\sum}}d_n{z_1}^n \quad  \text{and} \quad g(\xi_1, z_2)	=\overset{\infty}{\underset{m=-\infty}{\sum}}e_{-m}{z_2}^m,
\end{align}
respectively. By Cauchy's estimate in one-variable, we have that
\begin{align*}
	|d_{-n}|= \left|  \overset{\infty}{\underset{m=-\infty}{\sum}}\frac{a_{-n,m}} {\xi_2^m}\right| \leq \frac{\|g\|_{\infty, \CA_r^2}}{r^{|n|}} \quad \text{and} \quad |e_{-m}|=\left| \overset{\infty}{\underset{n=-\infty}{\sum}}a_{n,m} \xi_1^n\right| \leq \frac{\|g\|_{\infty, \CA_r^2}}{r^{|m|}}
\end{align*}
for $n, m \in \mathbb{Z}$. It is easy to see that
\begin{align*}
		\frac{1}{2\pi} \ \overset{2\pi}{\underset{0}{\int}}\left|g\left(\frac{1}{re^{i\zeta}}, \frac{1}{\xi_2}\right)\right|^2 d \zeta=\overset{\infty}{\underset{n=-\infty}{\sum}}|d_{-n}|^2r^{2n} \quad \text{and} \quad 	\frac{1}{2\pi} \ \overset{2\pi}{\underset{0}{\int}}\left|g\left(\xi_1, re^{i\zeta}\right)\right|^2 d\zeta=\overset{\infty}{\underset{m=-\infty}{\sum}}|e_{-m}|^2r^{2m}
\end{align*}
and so, we have
\begin{align}\label{eqn_0310}
		\left(\overset{\infty}{\underset{n=-\infty}{\sum}}|d_{-n}|^2r^{2n}\right)^{1\slash 2} \leq \|g\|_{\infty, \CA_r^2} \quad \text{and} \quad \left(\overset{\infty}{\underset{m=-\infty}{\sum}}|e_{-m}|^2r^{2m}\right)^{1\slash 2} \leq \|g\|_{\infty, \CA_r^2}.
\end{align}
Note that 
\begin{align}\label{eqn_0311}
	\left|\overset{\infty}{\underset{n=0}{\sum}} \  \overset{\infty}{\underset{m=-\infty}{\sum}}a_{n,-m} \xi_2^{m} \xi_1^n\right|
	&=	\left|\overset{\infty}{\underset{n=-\infty}{\sum}} \  \overset{\infty}{\underset{m=-\infty}{\sum}}\frac{a_{n,m}\xi_1^n}{\xi_2^{m}}-\overset{-1}{\underset{n=-\infty}{\sum}} \  \overset{\infty}{\underset{m=-\infty}{\sum}}\frac{a_{n,m}\xi_1^n}{\xi_2^{m}}\right| \notag \\
	&=\left|g\left(\xi_1, \frac{1}{\xi_2}\right)-\overset{\infty}{\underset{n=1}{\sum}} \  \overset{\infty}{\underset{m=-\infty}{\sum}}\frac{a_{-n,m}}{\xi_1^n\xi_2^{m}}\right| \notag \\
	& \leq \|g\|_{\infty, \CA_r^2}+\left|\overset{\infty}{\underset{n=1}{\sum}} \frac{d_{-n}}{\xi_1^n}\right| \notag \\
	&  \leq \|g\|_{\infty, \CA_r^2}+ \sqrt{\left(\overset{\infty}{\underset{n=1}{\sum}}|d_{-n}|^2r^{2n}\right)\left(\overset{\infty}{\underset{n=1}{\sum}}\frac{1}{r^{4n}}\right)} \notag \\
	& \leq \|g\|_{\infty, \CA_r^2}\left(1+ \frac{1}{\sqrt{r^4-1}}\right),
	\end{align}
	where the last inequality follows from \eqref{eqn_0310}. Also, we have
\begin{align}\label{eqn_0314}
	\left|	\overset{\infty}{\underset{n=0}{\sum}} \ \overset{\infty}{\underset{m=0}{\sum}}\frac{a_{n,m}\xi_1^n}{\xi_2^m}
	\right|
	&=\left|\overset{\infty}{\underset{m=0}{\sum}} \ \overset{\infty}{\underset{n=-\infty}{\sum}}\frac{a_{n,m}\xi_1^n}{\xi_2^m}-\overset{\infty}{\underset{m=0}{\sum}} \ \overset{-1}{\underset{n=-\infty}{\sum}}\frac{a_{n,m}\xi_1^n}{\xi_2^m}
	\right| \notag  \\
& \leq \left|\overset{\infty}{\underset{m=0}{\sum}} \frac{e_{-m}}{\xi_2^m}\right|+\left|\overset{\infty}{\underset{m=0}{\sum}} \ \overset{\infty}{\underset{n=1}{\sum}}\frac{a_{-n,m}}{\xi_1^n\xi_2^m}
\right| \notag \\
	& \leq 	\sqrt{\left(\overset{\infty}{\underset{m=0}{\sum}}|e_{-m}|^2r^{2m}\right)\left(\overset{\infty}{\underset{m=0}{\sum}}\frac{1}{r^{4m}}\right)} +\overset{\infty}{\underset{m=0}{\sum}} \ \overset{\infty}{\underset{n=1}{\sum}}\frac{\|g\|_{\infty, \CA_r^2}}{r^{2n}r^{2m}}  \qquad \left[\text{as} \ |a_{n,m}| \leq \frac{\|g\|_{\infty, \CA_r^2}}{r^{|n|+|m|}}\right] \notag \\
	&=\|g\|_{\infty, \CA_r^2}\left[\frac{r^2}{\sqrt{r^4-1}}+\frac{r^2}{(r^2-1)^2}\right].
\end{align}
Combining the estimates from \eqref{eqn_0311} and \eqref{eqn_0314}, we finally have that
\begin{align*}
	|g_2(\xi_1, \xi_2)|=\left| \overset{\infty}{\underset{n=0}{\sum}}\ \overset{\infty}{\underset{m=-\infty}{\sum}}a_{n, -m} \xi_1^n\xi_2^m-\overset{\infty}{\underset{n=0}{\sum}}   \ \overset{\infty}{\underset{m=0}{\sum}}\frac{a_{n,m}}{\xi_2^m} \xi_1^n\right| 
	\leq \|g\|_{\infty, \CA_r^2}\left[1+\frac{1+r^2}{\sqrt{r^4-1}}+\frac{r^2}{(r^2-1)^2}\right] 
\end{align*}
and so, 
\begin{align}\label{eqn_0315}
	\|g_2\|_{\infty, r\mathbb{T} \times r\mathbb{T}} \leq \|g\|_{\infty, \CA_r^2}\left[1+\frac{1+r^2}{\sqrt{r^4-1}}+\frac{r^2}{(r^2-1)^2}\right] \ .
\end{align}

\medskip 

\noindent \textbf{(3) An estimate for $\|g_3\|_{\infty, r\mathbb{T} \times r\mathbb{T}}$}. 
Let $(\xi_1, \xi_2)=(re^{i\theta}, re^{i\phi})$ for some $\theta, \phi \in \mathbb{R}$. Define
\[
\omega_{-n}=\overset{\infty}{\underset{m=-\infty}{\sum}}a_{n,m} \xi_2^m \quad \text{and} \quad \gamma_{-m}=\overset{\infty}{\underset{n=-\infty}{\sum}}\frac{a_{n, m}}{\xi_1^n} \quad (n, m \in \mathbb{Z}).
\] 
The Laurent series representations of the maps $g(., \xi_2)$ and $g(\xi_1^{-1}, .)$ in $\CA_r$ can be written as
\begin{align*}	
	g(z_1, \xi_2)
	=\overset{\infty}{\underset{n=-\infty}{\sum}}\omega_{-n}z_1^n \quad \text{and} \quad g(\xi_1^{-1}, z_2)
	=\overset{\infty}{\underset{m=-\infty}{\sum}}\gamma_{-m}z_2^m,
\end{align*}
respectively. It follows from Cauchy's estimate in one-variable that
\begin{align*}
	|\omega_{-n}|= \left|  \overset{\infty}{\underset{m=-\infty}{\sum}}a_{n,m} \xi_2^m\right| \leq \frac{\|g\|_{\infty, \CA_r^2}}{r^{|n|}} \quad \text{and} \quad |\gamma_{-m}|=\left| \overset{\infty}{\underset{n=-\infty}{\sum}}\frac{a_{n,m}}{ \xi_1^n}\right| \leq \frac{\|g\|_{\infty, \CA_r^2}}{r^{|m|}}
\end{align*}
for $n, m \in \mathbb{Z}$. Some simple calculations show that
\begin{align*}
	\frac{1}{2\pi} \ \overset{2\pi}{\underset{0}{\int}}\left|g\left( re^{i\zeta}, \xi_2\right)\right|^2 d\zeta=\overset{\infty}{\underset{n=-\infty}{\sum}}|\omega_{-n}|^2r^{2n} \quad \text{and} \quad 	\frac{1}{2\pi} \ \overset{2\pi}{\underset{0}{\int}}\left|g\left(\frac{1}{\xi_1}, re^{i\zeta} \right)\right|^2 d \zeta=\overset{\infty}{\underset{m=-\infty}{\sum}}|\gamma_{-m}|^2r^{2m}
\end{align*}
and so, we have
\begin{align}\label{eqn_3013}
	\left(\overset{\infty}{\underset{n=-\infty}{\sum}}|\omega_{-n}|^2r^{2n}\right)^{1\slash 2} \leq \|g\|_{\infty, \CA_r^2} \quad \text{and} \quad \left(\overset{\infty}{\underset{m=-\infty}{\sum}}|\gamma_{-m}|^2r^{2m}\right)^{1\slash 2} \leq \|g\|_{\infty, \CA_r^2}.
\end{align}
A routine computation gives that
\begin{align}\label{eqn_3014}
	\left|\overset{\infty}{\underset{m=0}{\sum}} \  \overset{\infty}{\underset{n=-\infty}{\sum}}a_{-n, m} \xi_1^{n} \xi_2^m\right|
	&=	\left|\overset{\infty}{\underset{n=-\infty}{\sum}} \  \overset{\infty}{\underset{m=-\infty}{\sum}}\frac{a_{n,m}\xi_2^m}{\xi_1^{n}}-\overset{-1}{\underset{m=-\infty}{\sum}} \  \overset{\infty}{\underset{n=-\infty}{\sum}}\frac{a_{n,m}\xi_2^m}{\xi_1^{n}}\right| \notag \\
	&=\left|g\left(\frac{1}{\xi_1}, \xi_2 \right)-\overset{\infty}{\underset{m=1}{\sum}} \frac{\gamma_{-m}}{\xi_2^m}\right| \notag \\
	& \leq \|g\|_{\infty, \CA_r^2}+\left|\overset{\infty}{\underset{m=1}{\sum}} \frac{\gamma_{-m}}{\xi_2^m}\right| \notag \\
	&  \leq \|g\|_{\infty, \CA_r^2}+ \sqrt{\left(\overset{\infty}{\underset{m=1}{\sum}}|\gamma_{-m}|^2r^{2m}\right)\left(\overset{\infty}{\underset{m=1}{\sum}}\frac{1}{r^{4m}}\right)} \notag \\
	& \leq \|g\|_{\infty, \CA_r^2}\left(1+ \frac{1}{\sqrt{r^4-1}}\right),
\end{align}
where the last inequality follows from \eqref{eqn_3013}. Again, some routine calculations give that
\begin{align}\label{eqn_3015}
	\left|	\overset{\infty}{\underset{n=0}{\sum}} \ \overset{\infty}{\underset{m=0}{\sum}}\frac{a_{n,m}\xi_2^m}{\xi_1^n}
	\right|
	& =\left|\overset{\infty}{\underset{n=0}{\sum}} \ \overset{\infty}{\underset{m=-\infty}{\sum}}\frac{a_{n,m}\xi_2^m}{\xi_1^n}-\overset{\infty}{\underset{n=0}{\sum}} \ \overset{\infty}{\underset{m=1}{\sum}}\frac{a_{n,-m}}{\xi_1^n\xi_2^m}
	\right| \notag  \\
	& \leq \left|\overset{\infty}{\underset{n=0}{\sum}} \frac{\omega_{-n}}{\xi_1^n}\right|+\left|\overset{\infty}{\underset{n=0}{\sum}} \ \overset{\infty}{\underset{m=1}{\sum}}\frac{a_{n, -m}}{\xi_1^n\xi_2^m}
	\right| \notag \\
	& \leq 	\sqrt{\left(\overset{\infty}{\underset{n=0}{\sum}}|\omega_{-n}|^2r^{2n}\right)\left(\overset{\infty}{\underset{n=0}{\sum}}\frac{1}{r^{4n}}\right)} +\overset{\infty}{\underset{n=0}{\sum}} \ \overset{\infty}{\underset{m=1}{\sum}}\frac{\|g\|_{\infty, \CA_r^2}}{r^{2n}r^{2m}}  \qquad \left[\text{as} \ |a_{n,m}| \leq \frac{\|g\|_{\infty, \CA_r^2}}{r^{|n|+|m|}}\right] \notag \\
	&=\|g\|_{\infty, \CA_r^2}\left[\frac{r^2}{\sqrt{r^4-1}}+\frac{r^2}{(r^2-1)^2}\right].
\end{align}
It now follows from \eqref{eqn_3014} and \eqref{eqn_3015} that
\begin{align*}
	|g_3(\xi_1, \xi_2)|=\left| \overset{\infty}{\underset{m=0}{\sum}} \  \overset{\infty}{\underset{n=-\infty}{\sum}}a_{-n, m} \xi_1^{n} \xi_2^m-\overset{\infty}{\underset{n=0}{\sum}} \ \overset{\infty}{\underset{m=0}{\sum}}\frac{a_{n,m}\xi_2^m}{\xi_1^n}\right| 
	\leq \|g\|_{\infty, \CA_r^2}\left[1+\frac{1+r^2}{\sqrt{r^4-1}}+\frac{r^2}{(r^2-1)^2}\right] 
\end{align*}
and so, 
\begin{align}\label{eqn_3016}
	\|g_3\|_{\infty, r\mathbb{T} \times r\mathbb{T}} \leq \|g\|_{\infty, \CA_r^2}\left[1+\frac{1+r^2}{\sqrt{r^4-1}}+\frac{r^2}{(r^2-1)^2}\right] \ .
\end{align}

\medskip 
\noindent \textbf{(4) An estimate for $\|g_4\|_{\infty, r\mathbb{T} \times r\mathbb{T}}$}. Let $(\xi_1, \xi_2)=(re^{i\theta}, re^{i\phi})$ for some $\theta, \phi \in \mathbb{R}$. Define
\[
\nu_{-n}=\overset{\infty}{\underset{m=-\infty}{\sum}}\frac{a_{n,m}}{\xi_2^m} \quad \text{and} \quad \mu_{-m}=\overset{\infty}{\underset{n=-\infty}{\sum}}a_{-n, m} \xi_1^n \quad (n, m \in \mathbb{Z}).
\]  
The Laurent series representations of the maps $g(., \xi_2^{-1})$ and $g(\xi_1^{-1}, .)$ in $\CA_r$ can be written as
\begin{align*}	
	g(z_1, \xi_2^{-1})
	=\overset{\infty}{\underset{n=-\infty}{\sum}}\nu_{-n}z_1^n \quad \text{and} \quad g(\xi_1^{-1}, z_2)
	=\overset{\infty}{\underset{m=-\infty}{\sum}}\mu_{-m}z_2^m,
\end{align*}
respectively. It follows from Cauchy's estimate in one-variable that
\begin{align*}
	|\nu_{-n}|= \left| \overset{\infty}{\underset{m=-\infty}{\sum}}\frac{a_{n,m}} {\xi_2^m}\right| \leq \frac{\|g\|_{\infty, \CA_r^2}}{r^{|n|}} \quad \text{and} \quad |\mu_{-m}|=\left| \overset{\infty}{\underset{n=-\infty}{\sum}}a_{-n,m}\xi_1^n\right| \leq \frac{\|g\|_{\infty, \CA_r^2}}{r^{|m|}}
\end{align*}
for $n, m \in \mathbb{Z}$. A simple calculation shows that
\begin{align*}
	\frac{1}{2\pi} \ \overset{2\pi}{\underset{0}{\int}}\left|g\left( re^{i\zeta}, \frac{1}{\xi_2}\right)\right|^2 d\zeta=\overset{\infty}{\underset{n=-\infty}{\sum}}|\nu_{-n}|^2r^{2n} \quad \text{and} \quad 	\frac{1}{2\pi} \ \overset{2\pi}{\underset{0}{\int}}\left|g\left(\frac{1}{\xi_1}, re^{i\zeta} \right)\right|^2 d \zeta=\overset{\infty}{\underset{m=-\infty}{\sum}}|\mu_{-m}|^2r^{2m}
\end{align*}
and so, we have
\begin{align}\label{eqn_3017}
	\left(\overset{\infty}{\underset{n=-\infty}{\sum}}|\nu_{-n}|^2r^{2n}\right)^{1\slash 2} \leq \|g\|_{\infty, \CA_r^2} \quad \text{and} \quad \left(\overset{\infty}{\underset{m=-\infty}{\sum}}|\mu_{-m}|^2r^{2m}\right)^{1\slash 2} \leq \|g\|_{\infty, \CA_r^2}.
\end{align}
A routine computation gives that
\begin{align}\label{eqn_3018}
	\left|\overset{\infty}{\underset{n=1}{\sum}} \  \overset{\infty}{\underset{m=-\infty}{\sum}}a_{-n, -m} \xi_1^{n} \xi_2^m\right|
	&=	\left|\overset{\infty}{\underset{n=-\infty}{\sum}} \  \overset{\infty}{\underset{m=-\infty}{\sum}}\frac{a_{n,m}}{\xi_1^{n}\xi_2^m}-\overset{\infty}{\underset{n=0}{\sum}} \  \left(\overset{\infty}{\underset{m=-\infty}{\sum}}\frac{a_{n,m}}{\xi_2^m}\right)\frac{1}{\xi_1^n}\right| \notag \\
	&=\left|g\left(\frac{1}{\xi_1}, \frac{1}{\xi_2} \right)-\overset{\infty}{\underset{n=0}{\sum}} \frac{\nu_{-n}}{\xi_1^n}\right| \notag \\
	& \leq \|g\|_{\infty, \CA_r^2}+\left|\overset{\infty}{\underset{n=0}{\sum}} \frac{\nu_{-n}}{\xi_1^n}\right| \notag \\
	&  \leq \|g\|_{\infty, \CA_r^2}+ \sqrt{\left(\overset{\infty}{\underset{n=0}{\sum}}|\nu_{-n}|^2r^{2n}\right)\left(\overset{\infty}{\underset{n=0}{\sum}}\frac{1}{r^{4n}}\right)} \notag \\
	& \leq \|g\|_{\infty, \CA_r^2}\left(1+ \frac{r^2}{\sqrt{r^4-1}}\right),
\end{align}
where the last inequality follows from \eqref{eqn_3017}. Again, some simple calculations give that
\begin{align}\label{eqn_3019}
	\left|	\overset{\infty}{\underset{n=1}{\sum}} \ \overset{\infty}{\underset{m=0}{\sum}}\frac{a_{-n,m}\xi_1^n}{\xi_2^m}
	\right|
	& =\left|\overset{\infty}{\underset{n=-\infty}{\sum}} \ \overset{\infty}{\underset{m=0}{\sum}}\frac{a_{-n,m}\xi_1^n}{\xi_2^m}-\overset{\infty}{\underset{n=0}{\sum}} \ \overset{\infty}{\underset{m=0}{\sum}}\frac{a_{n,m}}{\xi_1^n\xi_2^m}
	\right| \notag  \\
	& \leq \left|\overset{\infty}{\underset{m=0}{\sum}} \frac{\mu_{-m}}{\xi_2^m}\right|+\left|\overset{\infty}{\underset{n=0}{\sum}} \ \overset{\infty}{\underset{m=0}{\sum}}\frac{a_{n, -m}}{\xi_1^n\xi_2^m}
	\right| \notag \\
	& \leq 	\sqrt{\left(\overset{\infty}{\underset{m=0}{\sum}}|\mu_{-m}|^2r^{2m}\right)\left(\overset{\infty}{\underset{m=0}{\sum}}\frac{1}{r^{4m}}\right)} +\overset{\infty}{\underset{n=0}{\sum}} \ \overset{\infty}{\underset{m=0}{\sum}}\frac{\|g\|_{\infty, \CA_r^2}}{r^{2n}r^{2m}}  \qquad \left[\text{as} \ |a_{n,m}| \leq \frac{\|g\|_{\infty, \CA_r^2}}{r^{|n|+|m|}}\right] \notag \\
	&=\|g\|_{\infty, \CA_r^2}\left[\frac{r^2}{\sqrt{r^4-1}}+\left(\frac{r^2}{r^2-1}\right)^2 \ \right],
\end{align}
where the inequality follows from \eqref{eqn_3017}. It now follows from \eqref{eqn_3018} and \eqref{eqn_3019} that
\begin{align*}
	|g_4(\xi_1, \xi_2)|=\left|\overset{\infty}{\underset{n=1}{\sum}} \  \overset{\infty}{\underset{m=-\infty}{\sum}}a_{-n, -m} \xi_1^{n} \xi_2^m-\overset{\infty}{\underset{n=1}{\sum}} \ \overset{\infty}{\underset{m=0}{\sum}}\frac{a_{-n,m}\xi_1^n}{\xi_2^m}\right| 
	\leq \|g\|_{\infty, \CA_r^2}\left[1+\frac{2r^2}{\sqrt{r^4-1}}+\left(\frac{r^2}{r^2-1}\right)^2 \ \right] 
\end{align*}
and so, 
\begin{align}\label{eqn_3020}
	\|g_4\|_{\infty, r\mathbb{T} \times r\mathbb{T}} \leq \|g\|_{\infty, \CA_r^2}\left[1+\frac{2r^2}{\sqrt{r^4-1}}+\left(\frac{r^2}{r^2-1}\right)^2 \ \right] \ .
\end{align}
Combining \eqref{eqn_0308}, \eqref{eqn_0315}, \eqref{eqn_3016} and \eqref{eqn_3020}, the desired conclusion follows.
\end{proof}

Proposition \ref{prop_estimate} can be viewed as a two-variable analog of the classical one-variable decomposition of a holomorphic function on $\overline{\A}_r$ into its analytic and principal parts, together with the corresponding supremum norm estimates for each part. We are now in a position to present the main result of this section.

\begin{thm}\label{thm_pair}
	Let $(T_1, T_2)$ be a commuting pair of operators in $Q\A_r$. Then
	\[
	\|g(T_1, T_2)\| \leq \left[4+4\left(\frac{r^2+1}{r^2-1}\right)^{1\slash 2}+\left(\frac{r^2+1}{r^2-1}\right)^2 \ \right] \|g\|_{\infty, \CA_r^2} 
	\]
	for all $g \in \text{Rat}(\CA_r^2)$.
\end{thm}

\begin{proof}
	Let $g \in \text{Rat}(\CA_r^2)$. We can represent $g$ in the Laurent series form given by
	\[
	g(z_1, z_2)=\overset{\infty}{\underset{n=-\infty}{\sum}}\ \overset{\infty}{\underset{m=-\infty}{\sum}}a_{n,m}\ z_1^n z_2^m=g_1(z_1, z_2)+g_2(z_1, z_2^{-1})+g_3(z_1^{-1}, z_2)+g_4(z_1^{-1}, z_2^{-1}),
	\]
	where $g_1, g_2, g_3, g_4$ are given as in Proposition \ref{prop_estimate}. Since $g$ is a holomorphic map on $\CA_r^2$, the functions $g_1, g_2, g_3$ and $g_4$ are holomorphic functions on the compact set $r\overline{\mathbb{D}} \times r\overline{\mathbb{D}}$. Moreover, $\|T_1\|, \|T_2\|, \|T_1^{-1}\|, \|T_2^{-1}\| \leq r$ implies that $\sigma_T(T_1, T_2), \sigma_T(T_1, T_2^{-1}), \sigma_T(T_1^{-1}, T_2)$ and $\sigma_T(T_1^{-1}, T_2^{-1})$ are contained in $r\overline{\mathbb{D}} \times r\overline{\mathbb{D}}$. So, we can write
	\[
	g(T_1, T_2)=g_1(T_1, T_2)+g_2(T_1, T_2^{-1})+g_3(T_1^{-1}, T_2)+g_4(T_1^{-1}, T_2^{-1}).
	\] 
	An application of Ando's inequality to $(T_1, T_2), (T_1, T_2^{-1}), (T_1^{-1}, T_2)$ and $(T_1^{-1}, T_2^{-1})$ gives that
	\begin{align*}
		\|g(T_1, T_2)\| 
		& \leq \|g_1(T_1, T_2)\|+\|g_2(T_1, T_2^{-1})\|+\|g_3(T_1^{-1}, T_2)\|+\|g_4(T_1^{-1}, T_2^{-1})\| \\
		& \leq \|g_1\|_{\infty, r\mathbb{T} \times r\mathbb{T}}+\|g_2\|_{\infty, r\mathbb{T} \times r\mathbb{T}}+\|g_3\|_{\infty, r\mathbb{T} \times r\mathbb{T}}+\|g_4\|_{\infty, r\mathbb{T} \times r\mathbb{T}}\\
		& \leq \left[4+\frac{4(r^2+1)}{\sqrt{r^4-1}}+\frac{1+2r^2+r^4}{(r^2-1)^2} \ \right] \|g\|_{\infty, \CA_r^2} \quad [\text{by Proposition \ref{prop_estimate}}]\\
		&= \left[4+4\left(\frac{r^2+1}{r^2-1}\right)^{1\slash 2}+\left(\frac{r^2+1}{r^2-1}\right)^2 \ \right] \|g\|_{\infty, \CA_r^2}.
	\end{align*}	
	The proof is now complete.
\end{proof}

\section{Polyannulus as a $K$-spectral set for doubly commuting operators in $Q\A_r$}\label{sec_polyannulus}

\noindent Needless to mention, the technique adopted to obtain Theorem \ref{thm_pair} crucially depends  on the dilation of commuting pairs of contractions to commuting pairs of unitaries. For more than two commuting contractions, no such dilation theorem holds in general. One particular class for which the dilation theorem holds is the class of doubly commuting contractions. In this section, we show that the polyannulus $\CA_r^n$ is a $K_n$-spectral set for doubly commuting operators in $Q\A_r$, where 
\begin{align}\label{eqn_401}
K_n=\left(\frac{3r^2-1}{r^2-1}\right)^n.
\end{align}
The key ingredient in the proof of this result is the dilation theorem for doubly commuting contractions. It is a well-known fact that for a doubly commuting tuple of contractions $(T_1, \dotsc, T_n)$ acting on a Hilbert space $\HS$, there exists a commuting tuple of unitaries $(U_1, \dotsc, U_n)$ on a Hilbert space $\mathcal{K}$ containing $\HS$ such that 
$
g(T_1, \dotsc, T_n)=P_{\HS}g(U_1, \dotsc, U_n)|_{\HS}
$
for all $g \in \text{Rat}(\overline{\mathbb{D}}^n)$. It now follows from the spectral mapping principle that
\[
\|g(T_1, \dotsc, T_n)\| \leq \|g(U_1, \dotsc, U_n)\| =\|g\|_{\infty, \sigma_T(U_1, \dotsc, U_n)} \leq \|g\|_{\infty, \mathbb{T}^n}
\]
for all $g \in \text{Rat}(\overline{\mathbb{D}}^n)$. The above inequality is commonly referred to as the von Neumann's inequality in the multivariable setting. We refer an interested reader to Section 9 of Chapter I in \cite{NagyFoias6} for the success of dilation for a doubly commuting family of contractions. We now prove the following proposition, which plays a central role in establishing the main result of this section.

\begin{prop}\label{prop_401}
	Let $g$ be a holomorphic function on the polyannulus $\CA_r^n$. Then we have a decomposition of $g$ into $2^n$ functions given by
	\begin{align}\label{eqn_g2n} 
	g(z_1, \dotsc, z_n)=\underset{\mu}{\sum} \ g_{\mu}\left(z_1^{\mu(1)}, \dotsc, z_n^{\mu(n)}\right) \qquad \text{on \ $\CA_r^n$},
	\end{align}
where the sum varies over all functions $\mu: \{1, \dotsc, n\} \to \{1, -1\}$. Moreover, each $g_{\mu}$ is a holomorphic function on the polydisc $(r\overline{\mathbb{D}})^n$ and 
\begin{align}\label{eqn_g2n_esti}
\|g_{\mu}\|_{\infty, (r\mathbb{T})^n} \leq
\left(\frac{r^2}{r^2-1}\right)^{t}\left(\frac{2r^2-1}{r^2-1}\right)^{n-t}\|g\|_{\infty, \CA_r^n},
\end{align}
where $t \in \{0, \dotsc, n\}$ is the cardinality of the set $\mu^{-1}(\{1\})$.
\end{prop}

\begin{proof} The argument follows techniques similar to the proof of Proposition \ref{prop_estimate}. As in the biannulus case, we decompose the Laurent series of a holomorphic function on $\CA_r^n$ into $2^n$ parts. The estimates for these components are then obtained through repeated applications of the one-variable Cauchy estimates. However, Proposition \ref{prop_estimate} is based on methods due to Shields \cite{Shields}. The higher-dimensional setting requires an iterative scheme, leading to estimates different from those in Proposition \ref{prop_estimate} when $n=2$. Let $g$ be a holomorphic function on $\CA_r^n$. The Laurent series representation of $g$ is given by
	\begin{align}\label{eqn_gsplit1}
	g(z_1, \dotsc, z_n)=\overset{\infty}{\underset{\nu_1=-\infty}{\sum}}\dotsc \overset{\infty}{\underset{\nu_n=-\infty}{\sum}}a_{\nu_1,\dotsc, \nu_n}z_1^{\nu_1}\dotsc z_n^{\nu_n}
	\end{align}
for $(z_1, \dotsc, z_n) \in \CA_r^n$. The coefficients $a_{\nu_1, \dotsc, \nu_n}$ are given by
\[
a_{\nu_1, \dotsc, \nu_n}=\frac{1}{(2\pi i)^n}\int_{|\zeta_1|=\rho_1}\dotsc \int_{|\zeta_n|=\rho_n} \ \frac{g(\zeta_1, \dotsc, \zeta_n)}{\zeta_1^{\nu_1+1}\dotsc \zeta_n^{\nu_n+1}} \ d\zeta_1\dotsc d\zeta_n,
\]
where $\rho_1, \dotsc, \rho_n$ are scalars between $1\slash r$ and  $r$ (see Chapter II in \cite{Range}). It is evident that
\[
\left|a_{\nu_1, \dotsc, \nu_n}\right| \leq \frac{\|g\|_{\infty, \CA_r^n}}{r^{|\nu_1|+\dotsc +|\nu_n|}} \qquad \text{for all \ $\nu_1, \dotsc, \nu_n \in \mathbb{Z}$}.
\]
We can split $g(z_1, \dotsc, z_n)$ as in \eqref{eqn_gsplit1} into $2^n$ parts depending on non-negativity or negativity of the exponents $\nu_1, \dotsc, \nu_n$ of the monomials $z_1, \dotsc, z_n$, respectively. In other words, we have that
	\begin{align*}
	g(z_1, \dotsc, z_n)=\underset{\mu}{\sum} \ g_{\mu}\left(z_1^{\mu(1)}, \dotsc, z_n^{\mu(n)}\right) \qquad \text{on \ $\CA_r^n$},
\end{align*}
where the sum varies over all functions $\mu: \{1, \dotsc, n\} \to \{1, -1\}$. The functions $g_{\mu}$ are given by 
\[
g_\mu(w_1, \dotsc, w_n)=\overset{\infty}{\underset{\nu_1=\sigma(1)}{\sum}}\dotsc \overset{\infty}{\underset{\nu_n=\sigma(n)}{\sum}}a_{\mu(1)\nu_1,\dotsc, \mu(n)\nu_n}w_1^{\nu_1}\dotsc w_n^{\nu_n}, \ \  \text{where} \ \ 
\sigma(k)= \left\{
\begin{array}{ll}
	0 & \mbox{ if } \; \mu(k)=1 \\
	1 &  \mbox{ if} \; \mu(k)=-1 \\
\end{array} 
\right.
\]
for $1 \leq k \leq n$. It is clear that the above series converges uniformly and absolutely on the polydisc $(r\overline{\mathbb{D}})^n$. Thus, each $g_\mu$ defines a holomorphic function on $(r\overline{\mathbb{D}})^n$. It only remains to show the inequality as in \eqref{eqn_g2n_esti}. To make the algorithm clear to the readers, we first discuss the cases when $n=1, 2$ and then, building on $n=2$ case, we prove the desired conclusion for $n=3$. The general induction method follows the same techniques. 

\medskip 	
	
\noindent \textbf{Case $n=1$.} For a holomorphic function $g$ on $\CA_r$, the Laurent series representation of $g$ is given by 
\[
g(z)=\overset{\infty}{\underset{n=-\infty}{\sum}}a_nz^n=g_1(z)+g_2(1\slash z), \quad \text{where} \qquad g_1(z)=\overset{\infty}{\underset{n=0}{\sum}}a_nz^n \quad \text{and} \quad  g_2(z)=\overset{\infty}{\underset{n=1}{\sum}}a_{-n}z^n 
\]
for $z \in \CA_r$. Let $\xi=re^{i\theta}$ for $\theta \in \mathbb{R}$. By Cauchy's estimate in one-variable, we have that
\begin{align}\label{eqn_402}
\left|g_1(\xi)\right|=\left|g(\xi)-\overset{\infty}{\underset{n=1}{\sum}}\frac{a_{-n}}{\xi^n}\right| \leq \|g\|_{\infty, \CA_r}+\overset{\infty}{\underset{n=1}{\sum}}\frac{|a_{-n}|}{r^n} \leq \|g\|_{\infty, \CA_r}+\overset{\infty}{\underset{n=1}{\sum}}\frac{\|g\|_{\infty, \CA_r}}{r^{2n}}=\|g\|_{\infty, \CA_r}\left(\frac{r^2}{r^2-1}\right)  \notag \\ 
\end{align}
and
\begin{align}\label{eqn_403}
	\left|g_2(\xi)\right|=\left|g(1\slash \xi)-\overset{\infty}{\underset{n=0}{\sum}}\frac{a_{n}}{\xi^n}\right| \leq \|g\|_{\infty, \CA_r}+\overset{\infty}{\underset{n=0}{\sum}}\frac{|a_{n}|}{r^n} \leq \|g\|_{\infty, \CA_r}+\overset{\infty}{\underset{n=0}{\sum}}\frac{\|g\|_{\infty, \CA_r}}{r^{2n}}=\|g\|_{\infty, \CA_r}\left(\frac{2r^2-1}{r^2-1}\right). \notag \\
\end{align}

\noindent \textbf{Case $n=2$.} Let $g(z_1, z_2)$ be a holomorphic function on the biannulus $\CA_r^2=\CA_r \times \CA_r$. Following the discussion in the beginning of Section \ref{sec_biannulus}, we can write 
\begin{align*}
g(z_1, z_2)
=g_1(z_1, z_2)+g_2(z_1, z_2^{-1})+g_3(z_1^{-1}, z_2)+g_4(z_1^{-1}, z_2^{-1}),
\end{align*}
where 
\begin{align*}
	& g_1(z_1, z_2)=\overset{\infty}{\underset{n=0}{\sum}}\ \overset{\infty}{\underset{m=0}{\sum}}a_{n,m} \ z_1^nz_2^m, \quad \qquad g_2(z_1, w_2)=\overset{\infty}{\underset{n=0}{\sum}}\ \overset{\infty}{\underset{m=1}{\sum}}a_{n,-m}z_1^nw_2^m \\
	& g_3(w_1, z_2)=\overset{\infty}{\underset{n=1}{\sum}}\ \overset{\infty}{\underset{m=0}{\sum}}a_{-n,m} w_1^nz_2^m, \qquad g_4(w_1, w_2)=\overset{\infty}{\underset{n=1}{\sum}}\ \overset{\infty}{\underset{m=1}{\sum}}a_{-n,-m}w_1^nw_2^m.
\end{align*}
We compute the bounds for $\|g_j\|_{\infty, r\mathbb{T} \times r\mathbb{T}}$ for $1 \leq j \leq 4$ using Cauchy's estimates.

\medskip 

\noindent \textbf{(1) An estimate for $\|g_1\|_{\infty, r\mathbb{T} \times r\mathbb{T}}$}. Let $(\xi_1, \xi_2)=(re^{i\theta}, re^{i\phi})$ for some $\theta, \phi \in \mathbb{R}$.  Note that
\begin{align}\label{eqn_302}
	g(z_1, \xi_2)=\overset{\infty}{\underset{n=0}{\sum}}\left[ \overset{\infty}{\underset{m=-\infty}{\sum}}a_{n,m} \xi_2^m\right] z_1^n+\overset{\infty}{\underset{n=1}{\sum}}\left[ \overset{\infty}{\underset{m=-\infty}{\sum}}a_{-n,m} \xi_2^m\right] \frac{1}{z_1^n},
\end{align}
which is a Laurent series representation of $g(., \xi_2)$ in $\CA_r$. By Cauchy's estimate in one-variable, we have that
\begin{align}\label{eqn_303}
\left| \overset{\infty}{\underset{m=-\infty}{\sum}}a_{n,m} \xi_2^m\right| \leq \frac{\|g\|_{\infty, \CA_r^2}}{r^n},
\qquad  \left|  \overset{\infty}{\underset{m=-\infty}{\sum}}a_{-n,m} \xi_2^m\right| \leq \frac{\|g\|_{\infty, \CA_r^2}}{r^n}
\end{align}
and thus, we have by \eqref{eqn_302} \& \eqref{eqn_303} that
\begin{align}\label{eqn_304}
\left|	\overset{\infty}{\underset{n=0}{\sum}} \ \overset{\infty}{\underset{m=-\infty}{\sum}}a_{n,m} \xi_1^n \xi_2^m\right|
&= \left|g(\xi_1, \xi_2)-\overset{\infty}{\underset{n=1}{\sum}} \ \overset{\infty}{\underset{m=-\infty}{\sum}}a_{-n,m} \frac{\xi_2^m}{\xi_1^n}\right| \leq \|g\|_{\infty, \CA_r^2}\left(1+\overset{\infty}{\underset{n=1}{\sum}}\frac{1}{r^{2n}}\right)=\frac{r^2\|g\|_{\infty, \CA_r^2}}{r^2-1}.
\end{align}
Also, note that
\begin{align}\label{eqn_305}
	g(\xi_1, z_2)=\overset{\infty}{\underset{m=0}{\sum}}\left[ \overset{\infty}{\underset{n=-\infty}{\sum}}a_{n,m} \xi_1^n\right] z_2^m+\overset{\infty}{\underset{m=1}{\sum}}\left[ \overset{\infty}{\underset{n=-\infty}{\sum}}a_{n,-m} \xi_1^n\right] \frac{1}{z_2^m},
\end{align}
which is a Laurent series representation of the function $g(\xi_1, .)$ in $\CA_r$. An application of Cauchy's estimate from one-variable theory gives that
\begin{align}\label{eqn_306}
	\left| \overset{\infty}{\underset{n=-\infty}{\sum}}a_{n,m} \xi_1^n\right| \leq \frac{\|g\|_{\infty, \CA_r^2}}{r^m},
	\qquad  \left|  \overset{\infty}{\underset{n=-\infty}{\sum}}a_{n,-m} \xi_1^n\right| \leq \frac{\|g\|_{\infty, \CA_r^2}}{r^m}.
\end{align}
We have that
\begin{align}\label{eqn_307}
	\left|\overset{\infty}{\underset{n=0}{\sum}} \left( \overset{\infty}{\underset{m=1}{\sum}}\frac{a_{n,-m}}{\xi_2^m}\right)\xi_1^n\right| 
	&=\left|\overset{\infty}{\underset{m=1}{\sum}} \left( \overset{\infty}{\underset{n=-\infty}{\sum}} a_{n,-m}\xi_1^n\right)\frac{1}{\xi_2^m}-	\overset{\infty}{\underset{n=1}{\sum}} \ \overset{\infty}{\underset{m=1}{\sum}}\frac{a_{-n,-m}}{\xi_1^n \xi_2^m}\right|\\
& \leq	\overset{\infty}{\underset{m=1}{\sum}} \left| \overset{\infty}{\underset{n=-\infty}{\sum}} a_{n,-m}\xi_1^n\right|\frac{1}{|\xi_2^m|} +\overset{\infty}{\underset{n=1}{\sum}} \ \overset{\infty}{\underset{m=1}{\sum}}\left|\frac{a_{-n,-m}}{\xi_1^n \xi_2^m}\right| \notag  \\
& \leq 	\overset{\infty}{\underset{m=1}{\sum}} \frac{\|g\|_{\infty, \CA_r^2}}{r^{2m}} +\overset{\infty}{\underset{n=1}{\sum}} \ \overset{\infty}{\underset{m=1}{\sum}}\frac{\|g\|_{\infty, \CA_r^2}}{r^{2n}r^{2m}}  \qquad \left[\text{by \eqref{eqn_306}} \ \& \ |a_{n,m}| \leq \frac{\|g\|_{\infty, \CA_r^2}}{r^{|n|+|m|}}\right] \notag \\
&=\|g\|_{\infty, \CA_r^2}\left(\frac{r}{r^2-1}\right)^2.
\end{align}
Combining the above estimates, we finally have that
\begin{align*}
	|g_1(\xi_1, \xi_2)|
		&=\left| \overset{\infty}{\underset{n=0}{\sum}}\ \overset{\infty}{\underset{m=-\infty}{\sum}}a_{n,m} \xi_1^n\xi_2^m-\overset{\infty}{\underset{n=0}{\sum}}  \ \overset{\infty}{\underset{m=1}{\sum}}\frac{a_{n,-m}}{\xi_2^m} \xi_1^n\right| 
	 \leq \|g\|_{\infty, \CA_r^2}\left[\left(\frac{r^2}{r^2-1}\right) + \left(\frac{r}{r^2-1}\right)^2\right],
	 \end{align*}
where the last inequality follows from \eqref{eqn_304} and \eqref{eqn_307}. Hence, we have
\begin{align}\label{eqn_308}
\|g_1\|_{\infty, r\mathbb{T} \times r\mathbb{T}} \leq \left(\frac{r^2}{r^2-1}\right)^2\|g\|_{\infty, \CA_r^2} \ .
\end{align}

\medskip 

\noindent \textbf{(2) An estimate for $\|g_2\|_{\infty, r\mathbb{T} \times r\mathbb{T}}$}. Let $(\xi_1, \xi_2)=(re^{i\theta}, re^{i\phi})$ for some $\theta, \phi \in \mathbb{R}$.  Note that
\begin{align}\label{eqn_309}
	g(z_1, \xi_2^{-1})
	=\overset{\infty}{\underset{n=0}{\sum}}\left[ \overset{\infty}{\underset{m=-\infty}{\sum}}a_{n,-m} \xi_2^{m}\right] z_1^n+\overset{\infty}{\underset{n=1}{\sum}}\left[ \overset{\infty}{\underset{m=-\infty}{\sum}}a_{-n,-m} \xi_2^m\right] \frac{1}{z_1^n},
\end{align}
which is again a Laurent series representation of the function $g(., \xi_2^{-1})$ in $\CA_r$. By Cauchy's estimate in one-variable, we have that
\begin{align}\label{eqn_310}
	\left| \overset{\infty}{\underset{m=-\infty}{\sum}}a_{n,-m} \xi_2^{m}\right|, \left|  \overset{\infty}{\underset{m=-\infty}{\sum}}a_{-n,-m} \xi_2^m\right| \leq \frac{\|g\|_{\infty, \CA_r^2}}{r^n}
\end{align}
for $n, m \in \mathbb{Z}$.
It follows from \eqref{eqn_309} \& \eqref{eqn_310} that
\begin{align}\label{eqn_311}
	\left|\overset{\infty}{\underset{n=0}{\sum}} \  \overset{\infty}{\underset{m=-\infty}{\sum}}a_{n,-m} \xi_2^{m} \xi_1^n\right|
	=\left|g(\xi_1, \xi_2^{-1})-\overset{\infty}{\underset{n=1}{\sum}}\ \overset{\infty}{\underset{m=-\infty}{\sum}}a_{-n,-m} \frac{ \xi_2^m}{\xi_1^n}\right|
	 \leq \|g\|_{\infty, \CA_r^2}\left(1+\overset{\infty}{\underset{n=1}{\sum}}\frac{1}{r^{2n}}\right)
	=\frac{r^2\|g\|_{\infty, \CA_r^2}}{r^2-1}.
\end{align}
An alternative description of $g(\xi_1, \xi_2^{-1})$, differing from \eqref{eqn_309}, is given by
\begin{align}\label{eqn_312}
	g(\xi_1, z_2^{-1})=\overset{\infty}{\underset{m=0}{\sum}}\left[ \overset{\infty}{\underset{n=-\infty}{\sum}}a_{n,m} \xi_1^n\right] z_2^{-m}+\overset{\infty}{\underset{m=1}{\sum}}\left[ \overset{\infty}{\underset{n=-\infty}{\sum}}a_{n,-m} \xi_1^n\right]z_2^m,
\end{align}
which is a Laurent series form of the map $z_2 \mapsto g(\xi_1, z_2^{-1})$ in $\CA_r$. By Cauchy's estimate from one-variable theory, we have
\begin{align}\label{eqn_313}
\left|\overset{\infty}{\underset{n=-\infty}{\sum}}a_{n,m} \xi_1^n \right|, \left|\overset{\infty}{\underset{n=-\infty}{\sum}}a_{n,-m} \xi_1^n\right| \leq \frac{\|g\|_{\infty, \CA_r^2}}{r^m}.
	\end{align} 
We also have that
\begin{align}\label{eqn_314}
	\left|	\overset{\infty}{\underset{n=0}{\sum}} \left(\overset{\infty}{\underset{m=0}{\sum}}\frac{a_{n,m}}{\xi_2^m}\right) \xi_1^n
	\right|
	& \leq	\overset{\infty}{\underset{m=0}{\sum}} \left| \overset{\infty}{\underset{n=-\infty}{\sum}} a_{n,m}\xi_1^n\right|\frac{1}{|\xi_2^m|} +\overset{\infty}{\underset{n=1}{\sum}} \ \overset{\infty}{\underset{m=0}{\sum}}\left|\frac{a_{-n,m}}{\xi_1^n \xi_2^m}\right| \notag  \\
	& \leq 	\overset{\infty}{\underset{m=0}{\sum}} \frac{\|g\|_{\infty, \CA_r^2}}{r^{2m}} +\overset{\infty}{\underset{n=1}{\sum}} \ \overset{\infty}{\underset{m=0}{\sum}}\frac{\|g\|_{\infty, \CA_r^2}}{r^{2n}r^{2m}}  \qquad \left[\text{by \eqref{eqn_313}} \ \& \ |a_{-n,m}| \leq \frac{\|g\|_{\infty, \CA_r^2}}{r^{|n|+|m|}}\right] \notag \\
	&=\|g\|_{\infty, \CA_r^2}\left(\frac{r^2}{r^2-1}\right)^2.
\end{align}
Combining the above estimates, we finally have that
\begin{align*}
	|g_2(\xi_1, \xi_2)|
	=\left| \overset{\infty}{\underset{n=0}{\sum}}\ \overset{\infty}{\underset{m=-\infty}{\sum}}a_{n, -m} \xi_1^n\xi_2^m-\overset{\infty}{\underset{n=0}{\sum}}   \ \overset{\infty}{\underset{m=0}{\sum}}\frac{a_{n,m}}{\xi_2^m} \xi_1^n\right|  \notag 
	 \leq \|g\|_{\infty, \CA_r^2}\left(\frac{r^2}{r^2-1}\right) + \|g\|_{\infty, \CA_r^2}\left(\frac{r^2}{r^2-1}\right)^2,
\end{align*}
where the last inequality follows from \eqref{eqn_311} and \eqref{eqn_314}. Hence, we have 
\begin{align}\label{eqn_315}
	\|g_2\|_{\infty, r\mathbb{T} \times r\mathbb{T}} \leq \left(\frac{r^2}{r^2-1}\right)\left(\frac{2r^2-1}{r^2-1}\right)\|g\|_{\infty, \CA_r^2} \ .
\end{align}

\medskip 

\noindent \textbf{(3) An estimate for $\|g_3\|_{\infty, r\mathbb{T} \times r\mathbb{T}}$}. 
Let $(\xi_1, \xi_2)=(re^{i\theta}, re^{i\phi})$ for some $\theta, \phi \in \mathbb{R}$.  Note that
\begin{align*}	
	g(z_1^{-1}, \xi_2)
	=\overset{\infty}{\underset{n=0}{\sum}}\left[ \overset{\infty}{\underset{m=-\infty}{\sum}}a_{-n,m} \xi_2^{m}\right] z_1^n+\overset{\infty}{\underset{n=1}{\sum}}\left[ \overset{\infty}{\underset{m=-\infty}{\sum}}a_{n,m} \xi_2^m\right] \frac{1}{z_1^n},
\end{align*}
which is a Laurent series representation of the function $z_1 \mapsto g(z_1^{-1}, \xi_2)$ in $\CA_r$. We also have that 
\begin{align*}	
	g(\xi_1^{-1}, z_2)
	=\overset{\infty}{\underset{m=0}{\sum}}\left[ \overset{\infty}{\underset{n=-\infty}{\sum}}a_{-n,m} \xi_1^{n}\right] z_2^m+\overset{\infty}{\underset{m=1}{\sum}}\left[ \overset{\infty}{\underset{n=-\infty}{\sum}}a_{-n,-m} \xi_1^n\right] \frac{1}{z_2^m},
\end{align*}
which is again a Laurent series representation of the function $g(\xi_1^{-1}, .)$ in $\CA_r$. Using computations similar to those leading to \eqref{eqn_315} for $g(z_1^{-1}, \xi_2)$ and $g(\xi_1^{-1}, z_2)$, we obtain that
\begin{align}\label{eqn_316}
	\|g_3\|_{\infty, r\mathbb{T} \times r\mathbb{T}} \leq \left(\frac{r^2}{r^2-1}\right)\left(\frac{2r^2-1}{r^2-1}\right)\|g\|_{\infty, \CA_r^2} \ .
\end{align}

\medskip 
\noindent \textbf{(4) An estimate for $\|g_4\|_{\infty, r\mathbb{T} \times r\mathbb{T}}$}. Following the same approach as in the cases leading to \eqref{eqn_315} and \eqref{eqn_316}, we have
\begin{align}\label{eqn_317}
	\|g_4\|_{\infty, r\mathbb{T} \times r\mathbb{T}} \leq \left(\frac{2r^2-1}{r^2-1}\right)^2\|g\|_{\infty, \CA_r^2} \ .
\end{align}

\medskip 

\noindent \textbf{Case $n=3$.} Let $g(z_1, z_2, z_3)$ be a holomorphic function on $\CA_r^3$. Using the Laurent series representation of $g$ in $\CA_r^3$, we can write $g(z_1, z_2, z_3)$ as
\begin{align}\label{eqn_4021}
	& g_{+++}(z_1, z_2, z_3)+g_{++-}(z_1, z_2, z_3^{-1})+g_{+-+}(z_1, z_2^{-1}, z_3)+g_{+--}(z_1, z_2^{-1}, z_3^{-1})+g_{-++}(z_1^{-1}, z_2, z_3) \notag \\
	& +g_{-+-}(z_1^{-1}, z_2, z_3^{-1})+g_{--+}(z_1^{-1}, z_2^{-1}, z_3)+g_{---}(z_1^{-1}, z_2^{-1}, z_3^{-1}),\notag \\
\end{align}
where 
\begin{small} 
\begin{align}\label{eqn_4022}
	& g_{+++}(z_1, z_2, z_3)=\overset{\infty}{\underset{\nu_1=0}{\sum}}\ \overset{\infty}{\underset{\nu_2=0}{\sum}}\ \overset{\infty}{\underset{\nu_3=0}{\sum}}a_{\nu_1, \nu_2, \nu_3} \ z_1^{\nu_1}z_2^{\nu_2}z_3^{\nu_3}, \qquad  
	g_{++-}(z_1, z_2, z_3)=\overset{\infty}{\underset{\nu_1=0}{\sum}}\ \overset{\infty}{\underset{\nu_2=0}{\sum}}\ \overset{\infty}{\underset{\nu_3=1}{\sum}}a_{\nu_1, \nu_2, -\nu_3} \ z_1^{\nu_1}z_2^{\nu_2}z_3^{\nu_3} \notag \\
	& g_{+-+}(z_1, z_2, z_3)=\overset{\infty}{\underset{\nu_1=0}{\sum}}\ \overset{\infty}{\underset{\nu_2=1}{\sum}}\ \overset{\infty}{\underset{\nu_3=0}{\sum}}a_{\nu_1, -\nu_2, \nu_3} \ z_1^{\nu_1}z_2^{\nu_2}z_3^{\nu_3}, \ \ \quad  
	g_{+--}(z_1, z_2, z_3)=\overset{\infty}{\underset{\nu_1=0}{\sum}}\ \overset{\infty}{\underset{\nu_2=1}{\sum}}\ \overset{\infty}{\underset{\nu_3=1}{\sum}}a_{\nu_1, -\nu_2, -\nu_3} \ z_1^{\nu_1}z_2^{\nu_2}z_3^{\nu_3} \notag \\
	& g_{-++}(z_1, z_2, z_3)=\overset{\infty}{\underset{\nu_1=1}{\sum}}\ \overset{\infty}{\underset{\nu_2=0}{\sum}}\ \overset{\infty}{\underset{\nu_3=0}{\sum}}a_{-\nu_1, \nu_2, \nu_3} \ z_1^{\nu_1}z_2^{\nu_2}z_3^{\nu_3}, \qquad  
	g_{-+-}(z_1, z_2, z_3)=\overset{\infty}{\underset{\nu_1=1}{\sum}}\ \overset{\infty}{\underset{\nu_2=0}{\sum}}\ \overset{\infty}{\underset{\nu_3=1}{\sum}}a_{-\nu_1, \nu_2, -\nu_3} \ z_1^{\nu_1}z_2^{\nu_2}z_3^{\nu_3} \notag \\
	& g_{--+}(z_1, z_2, z_3)=\overset{\infty}{\underset{\nu_1=1}{\sum}}\ \overset{\infty}{\underset{\nu_2=1}{\sum}}\ \overset{\infty}{\underset{\nu_3=0}{\sum}}a_{-\nu_1, -\nu_2, \nu_3} \ z_1^{\nu_1}z_2^{\nu_2}z_3^{\nu_3}, \  \quad  
	g_{---}(z_1, z_2, z_3)=\overset{\infty}{\underset{\nu_1=1}{\sum}}\ \overset{\infty}{\underset{\nu_2=1}{\sum}}\ \overset{\infty}{\underset{\nu_3=1}{\sum}}a_{-\nu_1, -\nu_2, -\nu_3} \ z_1^{\nu_1}z_2^{\nu_2}z_3^{\nu_3}. \notag \\
	\end{align}
\end{small}
Thus, $g$ is decomposed into $2^3=8$ parts as is shown in \eqref{eqn_4021}. Each of these eight parts is explicitly defined in \eqref{eqn_4022}. For any non-negative integer $k$, $(z_i^{-1})^k$ stands for $\displaystyle \frac{1}{z_i^k}$ for $1 \leq i \leq 3$. We show that the supremum norm of each of this eight part  over $r\mathbb{T} \times r\mathbb{T} \times r\mathbb{T}$ is bounded above by 
\[
\left(\frac{r^2}{r^2-1}\right)^{t}\left(\frac{2r^2-1}{r^2-1}\right)^{3-t},
\]
where $t \in \{0, 1, 2, 3\}$ depending on how many monomials among $z_1, z_2, z_3$ appear with non-negative powers in the decomposition of $g$ as in \eqref{eqn_4021}. For example, $t=3$ corresponding to $g_{+++}$, $t=2$ corresponding to $g_{++-}, g_{+-+}, g_{-++}$, $t=1$ corresponding to $g_{+--}, g_{-+-}, g_{--+}$ and $t=0$ corresponding to $g_{---}$. Let $(\xi_1, \xi_2, \xi_3)=(re^{i\theta_1}, re^{i\theta_2}, re^{i\theta_3})$ for some $\theta_1, \theta_2, \theta_3 \in \mathbb{R}$. Then
\begin{align}\label{eqn_423}
&\left|g_{+++}(\xi_1, \xi_2, \xi_3)\right| \notag \\
&=\left|\overset{\infty}{\underset{\nu_1=0}{\sum}}\ \overset{\infty}{\underset{\nu_2=0}{\sum}}\ \overset{\infty}{\underset{\nu_3=0}{\sum}}a_{\nu_1, \nu_2, \nu_3} \ \xi_1^{\nu_1}\xi_2^{\nu_2}\xi_3^{\nu_3}\right| \notag \\
&=\left|\overset{\infty}{\underset{\nu_1=0}{\sum}}\ \overset{\infty}{\underset{\nu_2=0}{\sum}}\ \overset{\infty}{\underset{\nu_3=-\infty}{\sum}}a_{\nu_1, \nu_2, \nu_3} \ \xi_1^{\nu_1}\xi_2^{\nu_2}\xi_3^{\nu_3}-\overset{\infty}{\underset{\nu_1=0}{\sum}}\ \overset{\infty}{\underset{\nu_2=0}{\sum}}\ \overset{\infty}{\underset{\nu_3=1}{\sum}}a_{\nu_1, \nu_2, -\nu_3} \ \xi_1^{\nu_1}\xi_2^{\nu_2}\xi_3^{-\nu_3}\right| \notag \\
&\leq \left|\overset{\infty}{\underset{\nu_1=0}{\sum}}\ \overset{\infty}{\underset{\nu_2=0}{\sum}}\ \left(\overset{\infty}{\underset{\nu_3=-\infty}{\sum}}a_{\nu_1, \nu_2, \nu_3}\xi_3^{\nu_3}\right) \ \xi_1^{\nu_1}\xi_2^{\nu_2}\right|+\left|\ \overset{\infty}{\underset{\nu_3=1}{\sum}}\left(\overset{\infty}{\underset{\nu_1=0}{\sum}}\ \overset{\infty}{\underset{\nu_2=0}{\sum}}a_{\nu_1, \nu_2, -\nu_3}\ \xi_1^{\nu_1}\xi_2^{\nu_2}\right) \frac{1}{\xi_3^{\nu_3}}\right| \notag \\
& \leq \|g\|_{\infty, \CA_r^3}\left(\frac{r^2}{r^2-1}\right)^2+	\|g\|_{\infty, \CA_r^3}\left(\frac{r^2}{r^2-1}\right)^2\overset{\infty}{\underset{\nu_3=1}{\sum}}\frac{1}{r^{2\nu_3}} \qquad [\text{by $n=2$ case}] \notag \\
&\leq \|g\|_{\infty, \CA_r^3}\left(\frac{r^2}{r^2-1}\right)^3.
\end{align}
A similar computation as above shows that
\begin{align}\label{eqn_424}
	&\left|g_{++-}(\xi_1, \xi_2, \xi_3)\right| \notag \\
	&=\left|\overset{\infty}{\underset{\nu_1=0}{\sum}}\ \overset{\infty}{\underset{\nu_2=0}{\sum}}\ \overset{\infty}{\underset{\nu_3=1}{\sum}}a_{\nu_1, \nu_2, -\nu_3} \ z_1^{\nu_1}z_2^{\nu_2}z_3^{\nu_3}\right| \notag \\
	&\leq \left|\overset{\infty}{\underset{\nu_1=0}{\sum}}\ \overset{\infty}{\underset{\nu_2=0}{\sum}}\ \left(\overset{\infty}{\underset{\nu_3=-\infty}{\sum}}a_{\nu_1, \nu_2, -\nu_3}\xi_3^{\nu_3}\right) \ \xi_1^{\nu_1}\xi_2^{\nu_2}\right|+\left|\ \overset{\infty}{\underset{\nu_3=0}{\sum}}\left(\overset{\infty}{\underset{\nu_1=0}{\sum}}\ \overset{\infty}{\underset{\nu_2=0}{\sum}}a_{\nu_1, \nu_2, \nu_3}\ \xi_1^{\nu_1}\xi_2^{\nu_2}\right) \frac{1}{\xi_3^{\nu_3}}\right| \notag \\
	& \leq \|g\|_{\infty, \CA_r^3}\left(\frac{r^2}{r^2-1}\right)^2+	\|g\|_{\infty, \CA_r^3}\left(\frac{r^2}{r^2-1}\right)^2\overset{\infty}{\underset{\nu_3=0}{\sum}}\frac{1}{r^{2\nu_3}} \qquad [\text{by $n=2$ case}]\notag \\
	&\leq \|g\|_{\infty, \CA_r^3}\left(\frac{r^2}{r^2-1}\right)^2\left(\frac{2r^2-1}{r^2-1}\right).
\end{align}
It is not difficult to see that the above inequality holds when $\left|g_{++-}(\xi_1, \xi_2, \xi_3)\right|$ on the right hand side is replaced by either $\left|g_{+-+}(\xi_1, \xi_2, \xi_3)\right|$ or $	\left|g_{-++}(\xi_1, \xi_2, \xi_3)\right|$. Again following the similar techniques as above, we have that
\begin{align}\label{eqn_425}
	&\left|g_{--+}(\xi_1, \xi_2, \xi_3)\right|\notag \\
	&=\left|\overset{\infty}{\underset{\nu_1=1}{\sum}}\ \overset{\infty}{\underset{\nu_2=1}{\sum}}\ \overset{\infty}{\underset{\nu_3=0}{\sum}}a_{-\nu_1, -\nu_2, \nu_3} \ \xi_1^{\nu_1}\xi_2^{\nu_2}\xi_3^{\nu_3}\right| \notag \\
	&\leq \left|\overset{\infty}{\underset{\nu_1=1}{\sum}}\ \overset{\infty}{\underset{\nu_2=1}{\sum}}\ \left(\overset{\infty}{\underset{\nu_3=-\infty}{\sum}}a_{-\nu_1, -\nu_2, \nu_3}\xi_3^{\nu_3}\right) \ \xi_1^{\nu_1}\xi_2^{\nu_2}\right|+\left|\ \overset{\infty}{\underset{\nu_3=1}{\sum}}\left(\overset{\infty}{\underset{\nu_1=1}{\sum}}\ \overset{\infty}{\underset{\nu_2=1}{\sum}}a_{-\nu_1, -\nu_2, -\nu_3}\ \xi_1^{\nu_1}\xi_2^{\nu_2}\right) \frac{1}{\xi_3^{\nu_3}}\right|\notag \\
	& \leq \|g\|_{\infty, \CA_r^3}\left(\frac{2r^2-1}{r^2-1}\right)^2+	\|g\|_{\infty, \CA_r^3}\left(\frac{2r^2-1}{r^2-1}\right)^2\overset{\infty}{\underset{\nu_3=1}{\sum}}\frac{1}{r^{2\nu_3}} \qquad [\text{by $n=2$ case}]\notag\\
	&\leq \|g\|_{\infty, \CA_r^3}\left(\frac{2r^2-1}{r^2-1}\right)^2\left(\frac{r^2}{r^2-1}\right).
\end{align}
Evidently, the above inequality holds when $\left|g_{--+}(\xi_1, \xi_2, \xi_3)\right|$ on the right hand side is replaced by either $\left|g_{-+-}(\xi_1, \xi_2, \xi_3)\right|$ or $	\left|g_{+--}(\xi_1, \xi_2, \xi_3)\right|$. Finally, we also have that
\begin{align}\label{eqn_426}
	&\left|g_{---}(\xi_1, \xi_2, \xi_3)\right| \notag\\
	&=\left|\overset{\infty}{\underset{\nu_1=1}{\sum}}\ \overset{\infty}{\underset{\nu_2=1}{\sum}}\ \overset{\infty}{\underset{\nu_3=1}{\sum}}a_{-\nu_1, -\nu_2, -\nu_3} \ \xi_1^{\nu_1}\xi_2^{\nu_2}\xi_3^{\nu_3}\right|\notag\\
	&\leq \left|\overset{\infty}{\underset{\nu_1=1}{\sum}}\ \overset{\infty}{\underset{\nu_2=1}{\sum}}\ \left(\overset{\infty}{\underset{\nu_3=-\infty}{\sum}}a_{-\nu_1, -\nu_2, -\nu_3}\xi_3^{\nu_3}\right) \ \xi_1^{\nu_1}\xi_2^{\nu_2}\right|+\left|\ \overset{\infty}{\underset{\nu_3=0}{\sum}}\left(\overset{\infty}{\underset{\nu_1=1}{\sum}}\ \overset{\infty}{\underset{\nu_2=1}{\sum}}a_{-\nu_1, -\nu_2, \nu_3}\ \xi_1^{\nu_1}\xi_2^{\nu_2}\right) \frac{1}{\xi_3^{\nu_3}}\right|\notag \\
	& \leq \|g\|_{\infty, \CA_r^3}\left(\frac{2r^2-1}{r^2-1}\right)^2+	\|g\|_{\infty, \CA_r^3}\left(\frac{2r^2-1}{r^2-1}\right)^2\overset{\infty}{\underset{\nu_3=0}{\sum}}\frac{1}{r^{2\nu_3}} \qquad [\text{by $n=2$ case}] \notag\\
	&\leq \|g\|_{\infty, \CA_r^3}\left(\frac{2r^2-1}{r^2-1}\right)^3.
\end{align}
This establishes the case when $n=3$. The general conclusion can be obtained in a similar way. 
\end{proof}

Before presenting the main result of this section, we include the following remark, which compares two distinct estimates for $\|g(T_1, T_2)\|$, where $(T_1, T_2)$ is a commuting pair of operators in $Q\A_r$ and $g \in \text{Rat}(\CA_r^2)$. One of these estimates is derived using techniques based on the work of Shields \cite{Shields}, whereas the other is obtained by directly applying Cauchy's estimates.

\begin{rem} For a commuting pair $(T_1, T_2)$ of operators in $Q\A_r$, we have by Theorem \ref{thm_pair} that 
\begin{align}\label{eqn_430}
\|g(T_1, T_2)\| \leq \left[4+4\left(\frac{r^2+1}{r^2-1}\right)^{1\slash 2}+\left(\frac{r^2+1}{r^2-1}\right)^2 \ \right] \|g\|_{\infty, \CA_r^2} 
\end{align}
for all $g \in \text{Rat}(\CA_r^2)$. The estimate that we obtain in \eqref{eqn_430} is particularly based on the work of Shields \cite{Shields} mentioned as in Section \ref{sec_biannulus}. However, by employing Cauchy's estimates, we obtain a different bound in Proposition \ref{prop_401}. In the proof of Theorem \ref{thm_pair}, we split $g$ into four parts and found estimates for each of them. We achieve different estimates for those four parts of $g$ in Proposition \ref{prop_401} when considering the case $n=2$. Consequently, the $n=2$ case in Proposition \ref{prop_401} gives
\begin{align}\label{eqn_431}
	\|g(T_1, T_2)\|  \leq \left(\frac{3r^2-1}{r^2-1}\right)^2\|g\|_{\infty, \CA_r^2}
\end{align}	 
for all $g \in \text{Rat}(\CA_r^2)$. A routine computation shows that
\[
\left[4+4\left(\frac{r^2+1}{r^2-1}\right)^{1\slash 2}+\left(\frac{r^2+1}{r^2-1}\right)^2 \ \right] < \left(\frac{3r^2-1}{r^2-1}\right)^2.
\]
Consequently, the estimate given in \eqref{eqn_430} is finer than the one given in \eqref{eqn_431}. \qed
\end{rem}

We now present the main result of this section.

\begin{thm}\label{thm_mainII}
	Suppose $(T_1, \dotsc, T_n)$ is a doubly commuting tuple of operators in $Q\A_r$. Then
	\[
	\displaystyle \|g(T_1, \dotsc, T_n)\| \leq \left(\frac{3r^2-1}{r^2-1}\right)^n\|g\|_{\infty, \CA_r^n} 
	\]
for all $g \in \text{Rat}(\CA_r^n)$.	
\end{thm}

\begin{proof}
	Assume that $g \in \text{Rat}(\CA_r^n)$. Let $\mathcal{M}$ be the collection of all functions $\mu: \{1, \dotsc, n\} \to \{1, -1\}$ and let $|\mu^{-1}(\{1\})|$ be the cardinality of the set $\mu^{-1}\{1\}$. It follows from Proposition \ref{prop_401} that there is a decomposition of $g$ into $2^n$ functions given by
	\begin{align*} 
		g(z_1, \dotsc, z_n)=\underset{\mu \in \mathcal{M}}{\sum} \ g_{\mu}\left(z_1^{\mu(1)}, \dotsc, z_n^{\mu(n)}\right) \quad \text{on \ $\CA_r^n$}.
	\end{align*}
Moreover, each $g_{\mu}$ is a holomorphic function on the polydisc $(r\overline{\mathbb{D}})^n$. Since $\|T_j\|, \|T_j^{-1}\| \leq r$ for $1 \leq j \leq n$, the Taylor-joint spectrum of each of the $2^n$ doubly commuting tuples $\left(T_1^{\mu(1)}, \dotsc, T_n^{\mu(n)}\right)$ is contained in $(r\overline{\mathbb{D}})^n$. Now, an application of von Neumann's inequality on these $2^n$ doubly commuting operator tuples gives that
\begin{align*}
		\|g(T_1, \dotsc, T_n)\| 
		& \leq \underset{\mu \in \mathcal{M}}{\sum} \left\|g_{\mu}\left(T_1^{\mu(1)}, \dotsc, T_n^{\mu(n)}\right)\right\| \\
		& \leq \left[\underset{\mu \in \mathcal{M}}{\sum} \left(\frac{r^2}{r^2-1}\right)^{|\mu^{-1}(\{1\})|}\left(\frac{2r^2-1}{r^2-1}\right)^{n-|\mu^{-1}(\{1\})|}\right]\|g\|_{\infty, \CA_r^n} \qquad [\text{by Proposition \ref{prop_401}}]\\
		&=\left[\overset{n}{\underset{t=0}{\sum}}\binom{n}{t}\left(\frac{r^2}{r^2-1}\right)^{t}\left(\frac{2r^2-1}{r^2-1}\right)^{n-t}\right]\|g\|_{\infty, \CA_r^n}\\
		&=\left(\frac{3r^2-1}{r^2-1}\right)^n\|g\|_{\infty, \CA_r^n}.
\end{align*}	
The proof is now complete.
\end{proof}

\section{Bounds for optimal $K$-spectral constant for operator tuples in $Q\A_r$}\label{sec_lower}

\vspace{0.1cm} 

\noindent As mentioned in Section \ref{sec_intro}, the question of determining the smallest constant $K(\CA_r)$ for which $\CA_r$ is a $K(\CA_r)$-spectral set for every operator in \(Q\mathbb{A}_r\) is a challenging problem which still remains open. This problem has been investigated in the literature \cite{Badea, Pascoe, Shields, TsikalasII} (see also references therein). In this direction, Tsikalas \cite{TsikalasII} established a notable lower bound $K(\CA_r) \ge 2$, improving earlier estimates from \cite{Badea}. Moreover, it follows from Theorem \ref{thm_main} that
\[
2 \leq K(\CA_r) \le 2\left(1 + \frac{2r^2}{r^4 - 1}\right) \quad \text{and so,} \quad \lim_{r\to \infty}K(\CA_r)=2.
\]
In this section, we explore the multivariable analog of this result for commuting pairs and doubly commuting tuples of operators in $Q\A_r$. Suppose $K_{dc}(\CA_r^n)$ is the smallest (or optimal) constant for which the polyannulus $\CA_r^n$ is a $K_{dc}(\CA_r^n)$-spectral set for every doubly commuting operator $n$-tuple in $Q\A_r$.  It is not difficult to see that the function 
\[
\mathfrak{K}: (1, \infty) \to \mathbb{R}, \quad \text{given by} \quad \mathfrak{K}(r)= K_{dc}(\CA_r^n)
\]
is a non-increasing function. This property simply follows from the fact that $\CA_r \subseteq \CA_s$ for $1 < r \leq s$. Thus, $\lim_{r \to \infty} K_{dc}(\CA_r^n)$ exists (not necessarily finite). Our next result shows that this limit is indeed finite, and lies between $2^n$ and $3^n$.

\begin{thm}\label{thm:polyannulus}
	Let $K_{dc}(\CA_r^n)$ be the smallest constant for which the polyannulus $\CA_r^n$ is a $K_{dc}(\CA_r^n)$-spectral set for every doubly commuting operator $n$-tuple in $Q\A_r$. Then 
\[
2^n \leq K_{dc}(\CA_r^n) \leq \left(\frac{3r^2-1}{r^2-1}\right)^n.
\] 
Moreover, $\displaystyle 2^n \leq \lim_{r \to \infty} K_{dc}(\CA_r^n) \leq 3^n$.
\end{thm}

\begin{proof}
We recall the construction of a weighted bilateral shift due to Tsikalas \cite{TsikalasII}. To do so, consider the sequence $\{\xi_k\}_{k\in\Z}$ of positive numbers (weights) given by
	\begin{align*}
	\xi_{2\ell p+q} =r^q \quad \text{and} \quad 	\xi_{(2\ell+1)p+q}=r^{p-q} \ \ \text{for all} \quad q\in\{0,1,\dots, p\} \ \text{and} \ \ell \in \Z.	
	\end{align*}  
Consider the Hilbert space	consisting of sequences given by
	\[
	L^2(\xi)=\left\{\alpha=\{\alpha_k \}_{k\in\Z} \ :\ \|\alpha\|^2_{\xi}=\sum_{k\in\Z}|\alpha_k \xi_k|^2<\infty\right\},
	\]
with inner product $\displaystyle \langle \alpha, \gamma\rangle_{\xi}:=\sum_{k\in\Z} \alpha_k\overline{\gamma_k}\xi_k^2$. Let $S$ be the weighted bilateral shift on $L^2(\xi)$ given by $S\alpha=\{\alpha_{k-1}\}_{k \in \Z}$. For $\displaystyle g_m(z)=r^{-m}(z^m+z^{-m})$ with $m \in \mathbb{N}$, the proof of Theorem 1.1 in \cite{TsikalasII} shows
\begin{equation}\label{eqn_S}
\frac{\|g_m(S)\|}{\|g_m\|_{\infty, \CA_r}} \geq \frac{2r^m}{r^m+r^{-m}}.
\end{equation}
Consider the $n$-fold tensor product Hilbert space $\mathcal{H}= L^2(\xi) \otimes \cdots \otimes L^2(\xi)$. For $1 \leq j \leq n$, define 
	\[
	T_j := I \otimes \dots \otimes I \otimes \underbrace{S}_{j\text{-th factor}} \otimes I \otimes \dots \otimes I.
	\]
Clearly, each $T_j$ is invertible with $\|T_j\| = \|S\| \le r$ and $\|T_j^{-1}\| = \|S^{-1}\| \le r$. Moreover, $\underline{T}=(T_1, \dotsc, T_n)$ is a doubly commuting tuple since $T_i, T_j$ involve $S$ at different positions in their tensor product form for $i \ne j$. Now define a sequence of multivariable functions 
	\[
	f_m(z_1,\dots,z_n) := \prod_{j=1}^{n} g_m(z_j) \in \text{Rat}(\CA_r^n).
	\]
By the holomorphic functional calculus for commuting operators, we have
	\[
	f_m(T_1,\dots,T_n) = \prod_{j=1}^n g_m(T_j) = g_m(S) \otimes  \dots \otimes g_m(S) \ \ (n \text{ times}) \ \ \text{and so,} \ \ \|f_m(T_1,\dots,T_n)\|=\|g_m(S)\|^n.
	\]	
Since $\|f_m\|_{\infty, \CA_r^n} \leq  \prod_{j=1}^n \|g_m\|_{\infty, \CA_r^n} = \left(\|g_m\|_{\infty, \CA_r^n}\right)^n$, it follows from \eqref{eqn_S} that
\[
K_{dc}(\CA_r^n) \geq \frac{\|f_m(T_1,\dots,T_n)\|}{\|f_m\|_{\infty, \CA_r^n}} \geq  \left( \frac{\|g_m(S)\|}{\|g_m\|_{\CA_r^n}} \right)^n \geq \left(\frac{2r^m}{r^m+r^{-m}}\right)^n.
\]	
As $m \to \infty$, $K_{dc}(\CA_r^n) \geq 2^n$. We have by Theorem \ref{thm_mainII} that $\displaystyle K_{dc}(\CA_r^n) \leq \left(\frac{3r^2-1}{r^2-1}\right)^n$. From the discussion preceding the statement of the theorem, $\lim_{r \to \infty} K_{dc}(\CA_r^n)$ exists and the result follows.
\end{proof}

Moving forward, we establish a result analogous to Theorem \ref{thm:polyannulus} for a commuting pair of operators in $Q\A_r$. Let $K(\CA_r^2)$ be the smallest constant for which the biannulus $\CA_r^2$ is a $K(\CA_r^2)$-spectral set for every commuting pair of operators in $Q\A_r$. Since every doubly commuting tuple is a commuting one, we have that
$
\|f(T_1, T_2)\| \leq K(\CA_r^2) \|f\|_{\infty, \CA_r^2}
$
for every doubly commuting pair $(T_1, T_2)$ of operators in $Q\A_r$ and $f \in \text{Rat}(\CA_r^2)$. By minimality, we have that $K_{dc}(\CA_r^2) \leq K(\CA_r^2)$. It now follows from Theorem \ref{thm:polyannulus} that $2^2 \leq K_{dc}(\CA_r^2) \leq K(\CA_r^2)$. Finally, an upper bound on $K(\CA_r^2)$ is guaranteed by Theorem \ref{thm_pair}, leading to the following result.

\begin{thm}
If $K(\CA_r^2)$ is the smallest constant for which $\CA_r^2$ is a $K(\CA_r^2)$-spectral set for every commuting pair of operators in $Q\A_r$, then 
\[
2^2  \leq K(\CA_r^2) \leq \left[4+4\left(\frac{r^2+1}{r^2-1}\right)^{1\slash 2}+\left(\frac{r^2+1}{r^2-1}\right)^2 \ \right] \quad \text{and} \quad 2^2 \leq \lim_{r \to \infty} K(\CA_r^2) \leq 3^2.
\]
\end{thm}

\begin{proof}
The first inequality in the statement of the theorem follows from the above discussion and Theorem \ref{thm_pair}. It is easy to see that the map $r \mapsto K(\CA_r^2)$ is a non-increasing real-valued function on $(1, \infty)$. So, $\lim_{r \to \infty} K(\CA_r^2)$ exists, and the desired conclusion follows. 
\end{proof}

\vspace{0.1cm}

\noindent \textbf{Funding.} The first named author is supported in part by the ``Core Research Grant" with Award No. CRG/2023/005223 from Anusandhan National Research Foundation (ANRF) of Govt. of India. The third named author is supported via the IIT Bombay RDF Grant of the first named author with Project Code RI/0115-10001427.

\end{document}